\documentclass[11pt]{amsart}
\usepackage{amscd,amssymb}
\usepackage{amsthm,amsmath}
\usepackage[matrix,arrow]{xy}
\usepackage{graphicx}
\usepackage{enumerate,amssymb,setspace}
\usepackage{amsfonts,amsthm,amscd}
\usepackage{tikz}

\theoremstyle{definition}
\newtheorem{example}[equation]{Example}
\newtheorem{definition}[equation]{Definition}
\newtheorem{theorem}[equation]{Theorem}
\newtheorem{proposition}[equation]{Proposition}
\newtheorem{lemma}[equation]{Lemma}
\newtheorem{question}[equation]{Question}
\newtheorem{corollary}[equation]{Corollary}

\theoremstyle{remark}
\newtheorem{remark}[equation]{Remark}

\makeatletter\@addtoreset{equation}{section} \makeatother

\title{Worst singularities of plane curves of given degree}

\author{Ivan Cheltsov}



\begin{document}

\begin{abstract}
We prove that $\frac{2}{d}$, $\frac{2d-3}{(d-1)^2}$,
$\frac{2d-1}{d(d-1)}$, $\frac{2d-5}{d^2-3d+1}$ and
$\frac{2d-3}{d(d-2)}$ are the smallest log canonical thresholds of
reduced plane curves of degree $d\geqslant 3$, and we describe reduced
plane curves of degree $d$ whose log canonical thresholds are
these numbers. As an application, we prove that $\frac{2}{d}$, $\frac{2d-3}{(d-1)^2}$,
$\frac{2d-1}{d(d-1)}$, $\frac{2d-5}{d^2-3d+1}$ and
$\frac{2d-3}{d(d-2)}$  are the smallest values of the
$\alpha$-invariant of Tian of smooth surfaces in $\mathbb{P}^3$ of degree $d\geqslant 3$.
We also prove that every reduced plane curve of degree
$d\geqslant 4$ whose log canonical threshold is smaller than
$\frac{5}{2d}$ is GIT-unstable for the action of the group
$\mathrm{PGL}_3(\mathbb{C})$, and we describe GIT-semistable
reduced plane curves with log canonical thresholds~$\frac{5}{2d}$.
 \end{abstract}

\maketitle

All varieties are assumed to be algebraic, projective and defined
over $\mathbb{C}$.

\section{Introduction}
\label{section:intro}

Let $C_d$ be a \emph{reduced} plane curve in $\mathbb{P}^2$ of degree $d\geqslant 3$, and let $P$ be a point in $C_d$. 
The curve $C_d$ can have \emph{any} given plane curve singularity at $P$ provided that its degree $d$ is \emph{sufficiently} big. 
Thus, it is natural to ask

\begin{question}
\label{question:plane-curves-old} What is the \emph{worst} singularity that $C_d$ can have at $P$?
\end{question}

Denote by $m_P$ the multiplicity of the curve $C_d$ at the point
$P$, and denote by $\mu(P)$ the Milnor number of the point $P$. If
we use $m_P$ to measure the singularity of $C_d$ at the point $P$,
then a union of $d$ lines passing through $P$ is an answer to
Question~\ref{question:plane-curves-old}, since $m_P\leqslant d$,
and $m_P=d$ if and only if $C_d$ is a union of $d$ lines passing
through $P$. If we use the Milnor number $\mu(P)$, then the answer
would be the same, since $\mu(P)\leqslant(d-1)^2$, and
$\mu(P)=(d-1)^2$ if and only if $C_d$ is a union of $d$ lines
passing through $P$. Alternatively, we can use the number
$$
\mathrm{lct}_P\big(\mathbb{P}^2,C_d\big)=\mathrm{sup}\Big\{\lambda\in\mathbb{Q}\ \Big|\ \text{the log pair}\ \big(\mathbb{P}^2, \lambda C_d\big)\ \text{is log canonical at $P$}\Big\},%
$$
which is known as the \emph{log canonical threshold} of the log
pair $(\mathbb{P}^2, C_d)$ at the point $P$ or the log canonical
threshold of the curve $C_d$ at the point $P$ (see
\cite[Definition~6.34]{CoKoSm}). The smallest
$\mathrm{lct}_P(\mathbb{P}^2,C_d)$ when $P$ runs through all
points in $C_d$ is usually denoted by
$\mathrm{lct}(\mathbb{P}^2,C_d)$. Note that
$$
\frac{1}{m_P}\leqslant\mathrm{lct}_P\big(\mathbb{P}^2,C_d\big)\leqslant\frac{2}{m_P}.
$$
This is well-known (see, \cite[Exercise~6.18]{CoKoSm} and \cite[Lemma~6.35]{CoKoSm}). So, the smaller
$\mathrm{lct}_P(\mathbb{P}^2,C_d)$, the worse singularity of the
curve $C_d$ at the point $P$ is.

\begin{example}
\label{example:Kuwata} Suppose that $C_d$ is given by
$x_1^{n_1}x_2^{n_2}(x_1^{m_1}+x_2^{m_2})=0$
up to analytic change of local coordinates, where $m_1$ and
$m_2$ are non-negative integers, and $n_1,n_2\in\{0,1\}$. Then
$$
\mathrm{lct}_P\big(\mathbb{P}^2,C_d\big)=\mathrm{min}\Bigg\{1,\frac{\frac{1}{m_1}+\frac{1}{m_2}}{1+\frac{n_1}{m_1}+\frac{n_2}{m_2}}\Bigg\}%
$$
by \cite[Proposition~2.2]{Kuwata}.
\end{example}

Log canonical thresholds of plane curves have been intensively studied (see, for example, \cite{Kuwata}). 
Surprisingly, they give the same answer to Question~\ref{question:plane-curves-old} by

\begin{theorem}[{\cite[Theorem~4.1]{Che01}}]
\label{theorem:plane-curve-old} One has
$\mathrm{lct}_P(\mathbb{P}^2,C_d)\geqslant\frac{2}{d}$. Moreover,
\hbox{$\mathrm{lct}(\mathbb{P}^2,C_d)=\frac{2}{d}$} if and only if
$C_d$ is a union of $d$ lines that pass through $P$.
\end{theorem}

In this paper we want to address

\begin{question}
\label{question:plane-curves} What is the \emph{second worst}
singularity that $C_d$ can have at $P$?
\end{question}

To give a \emph{reasonable} answer to this question, we have to
disregard $m_P$ by obvious reasons. Thus, we will use the numbers
$\mu(P)$ and $\mathrm{lct}_P(\mathbb{P}^2,\mathbb{C}_d)$. For
cubic curves, they give the same answer.

\begin{example}
\label{example:plane-curve-cubic} Suppose that $d=3$, $m_P<3$ and
$P$ is a singular point of $C_3$. Then $P$ is a singular point of
type $\mathbb{A}_1$, $\mathbb{A}_2$ or $\mathbb{A}_3$. Moreover,
if $C_3$ has singularity of type $\mathbb{A}_3$ at $P$, then
$C_3=L+C_2$, where $C_2$ is a smooth conic, and $L$ is a line
tangent to $C_2$ at $P$. Furthermore, we have
$$
\mu(P)=\left\{%
\aligned
&1\ \text{if $C_3$ has $\mathbb{A}_1$ singularity at $P$},\\%
&2\ \text{if $C_3$ has $\mathbb{A}_2$ singularity at $P$},\\%
&3\ \text{if $C_3$ has $\mathbb{A}_3$ singularity at $P$}.\\%
\endaligned\right.%
$$
Similarly, we have
$$
\mathrm{lct}_P\big(\mathbb{P}^2,C_3\big)=\left\{%
\aligned
&1\ \text{if $C_3$ has $\mathbb{A}_1$ singularity at $P$},\\%
&\frac{5}{6}\ \text{if $C_3$ has $\mathbb{A}_2$ singularity at $P$},\\%
&\frac{3}{4}\ \text{if $C_3$ has $\mathbb{A}_3$ singularity at $P$}.\\%
\endaligned\right.%
$$
\end{example}

For quartic curves, the numbers $\mu(P)$ and
$\mathrm{lct}_P(\mathbb{P}^2,\mathbb{C}_d)$ give different answers
to Question~\ref{question:plane-curves}.

\begin{example}
\label{example:plane-curve-quartic} Suppose that $d=4$, $m_P<4$
and $P$ is a singular point of $C_4$. Going through the list of
all possible singularities that $C_P$ can have at $P$ (see, for
example, \cite{Hui}), we obtain
$$
\mu(P)=\left\{%
\aligned
&6\ \text{if $C_4$ has $\mathbb{D}_6$ singularity at $P$},\\%
&6\ \text{if $C_4$ has $\mathbb{A}_6$ singularity at $P$},\\%
&6\ \text{if $C_4$ has $\mathbb{E}_6$ singularity at $P$},\\%
&7\ \text{if $C_4$ has $\mathbb{A}_7$ singularity at $P$},\\%
&7\ \text{if $C_4$ has $\mathbb{E}_7$ singularity at $P$},\\%
\endaligned\right.%
$$
and $\mu(P)<6$ in all remaining cases. Similarly, we get
$$
\mathrm{lct}_P\big(\mathbb{P}^2,C_4\big)=\left\{%
\aligned
&\frac{5}{8}\ \text{if $C_4$ has $\mathbb{A}_7$ singularity at $P$},\\%
&\frac{5}{8}\ \text{if $C_4$ has $\mathbb{D}_5$ singularity at $P$},\\%
&\frac{3}{5}\ \text{if $C_4$ has $\mathbb{D}_6$ singularity at $P$},\\%
&\frac{7}{12}\ \text{if $C_4$ has $\mathbb{E}_6$ singularity at $P$},\\%
&\frac{5}{9}\ \text{if $C_4$ has $\mathbb{E}_7$ singularity at $P$},\\%
\endaligned\right.%
$$
and $\mathrm{lct}_P(\mathbb{P}^2,C_4)>\frac{5}{8}$ in all
remaining cases.
\end{example}

Recently, Arkadiusz P\l oski proved that
$\mu(P)\leqslant(d-1)^2-\lfloor\frac{d}{2}\rfloor$ provided that
$m_P<d$. Moreover, he described $C_d$ in the case when
$\mu(P)=(d-1)^2-\lfloor\frac{d}{2}\rfloor$. To present his
description, we need

\begin{definition}
\label{definition:Ploski-curve-even}
The curve $C_d$ is an \emph{even P\l oski} curve if $d$ is even,
the curve $C_d$ has $\frac{d}{2}\geqslant 2$ irreducible components that are smooth conics passing through $P$,
and all irreducible components of $C_d$ intersect each other pairwise at $P$ with multiplicity $4$.
The curve $C_d$ is an \emph{odd P\l oski} curve if $d$ is odd,
the curve $C_d$ has $\frac{d+1}{2}\geqslant 2$ irreducible components that all pass through $P$,
$\frac{d-1}{2}$ irreducible component of the curve $C_d$ are smooth conics that intersect each other pairwise at $P$
with multiplicity $4$, and the remaining irreducible component is a line in $\mathbb{P}^2$ that is tangent at $P$ to all other irreducible components.
We say that $C_d$ is \emph{P\l oski} curve if it is either an even P\l oski curve or an odd P\l oski curve.
\end{definition}

Each P\l oski curve has unique singular point. If $d=4$, then
$C_4$ is a P\l oski curve if and only if it has a singular point
of type $\mathbb{A}_7$. Thus, if $d=4$, then
$\mu(P)=(d-1)^2-\lfloor\frac{d}{2}\rfloor=7$ if and only if either
$C_4$ is a P\l oski curve and $P$ is its singular point or $C_4$
has singularity $\mathbb{E}_7$ at the point $P$ (see
Example~\ref{example:plane-curve-quartic}). For $d\geqslant 5$,
P\l oski proved

\begin{theorem}[{\cite[Theorem~1.4]{Ploski}}]
\label{theorem:Ploski} If $d\geqslant 5$, then
$\mu(P)=(d-1)^2-\lfloor\frac{d}{2}\rfloor$ if and only if $C_d$ is
a P\l oski curve and $P$ is its singular point.
\end{theorem}

This result gives a \emph{very good} answer to Question~\ref{question:plane-curves}.
The main goal of this paper is to give an  answer to Question~\ref{question:plane-curves}.
using log canonical thresholds.
Namely, we will prove that
$$
\mathrm{lct}_P\big(\mathbb{P}^2,C_d\big)\geqslant\frac{2d-3}{(d-1)^2}
$$
provided that $m_P<d$, and we will describe $C_d$ in the case when
$\mathrm{lct}_P(\mathbb{P}^2,C_d)=\frac{2d-3}{(d-1)^2}$. To present this
description, we need

\begin{definition}
\label{definition:plane-curve-singularities} The curve $C_d$ has
singularity of type $\mathbb{T}_r$ (resp., $\mathbb{K}_r$,
$\widetilde{\mathbb{T}}_{r}$, $\widetilde{\mathbb{K}}_{r}$) at the
point $P$ if the curve $C_d$ can be given by $x_1^r=x_1x_2^r$
(resp., $x_1^{r}=x_2^{r+1}$, $x_2x_1^{r-1}=x_1x_2^{r}$,
$x_2x_1^{r-1}=x_2^{r+1}$) up to analytic change of  coordinates at the point $P$.
\end{definition}

Note that $\mathbb{T}_2=\mathbb{A}_3$, $\mathbb{K}_2=\mathbb{A}_2$,
$\widetilde{\mathbb{T}}_{2}=\widetilde{\mathbb{K}}_{2}=\mathbb{A}_1$,
$\widetilde{\mathbb{K}}_{3}=\mathbb{D}_5$,
$\widetilde{\mathbb{T}}_{3}=\mathbb{D}_6$,
$\mathbb{K}_3=\mathbb{E}_6$ and $\mathbb{T}_3=\mathbb{E}_7$.
Furthermore, since we assume that $d\geqslant 3$, the formula in Example~\ref{example:Kuwata} gives
$$
\mathrm{lct}_P\big(\mathbb{P}^2,C_d\big)=\left\{%
\aligned
&\frac{2d-3}{(d-1)^2}\ \text{if $C_d$ has $\mathbb{T}_{d-1}$ singularity at $P$},\\%
&\frac{2d-1}{d(d-1)}\ \text{if $C_d$ has $\mathbb{K}_{d-1}$ singularity at $P$},\\%
&\frac{2d-5}{d^2-3d+1}\ \text{if $C_d$ has $\widetilde{\mathbb{T}}_{d-1}$ singularity at $P$},\\%
&\frac{2d-3}{d(d-2)}\ \text{if $C$ has $\widetilde{\mathbb{K}}_{d-1}$ singularity at $P$},\\%
\endaligned\right.%
$$
where $\frac{2}{d}<\frac{2d-3}{(d-1)^2}<\frac{2d-1}{d(d-1)}<\frac{2d-5}{d^2-3d+1}\leqslant\frac{2d-3}{d(d-2)}$.
In this paper we will prove

\begin{theorem}
\label{theorem:plane-curve} Suppose that $d\geqslant 4$ and $\mathrm{lct}_P(\mathbb{P}^2,C_d)\leqslant\frac{2d-3}{d(d-2)}$.
Then one of the following holds:
\begin{enumerate}
\item $m_P=d$,
\item the curve $C_d$ has singularity of type $\mathbb{T}_{d-1}$, $\mathbb{K}_{d-1}$, $\widetilde{\mathbb{T}}_{d-1}$ or $\widetilde{\mathbb{K}}_{d-1}$ at the point $P$,
\item $d=4$ and $C_d$ is a P\l oski quartic curve (in this case $\mathrm{lct}_P(\mathbb{P}^2,C_d)=\frac{5}{8}$).
\end{enumerate}
\end{theorem}

This result describes the \emph{five worst} singularities that $C_d$ can have at the point $P$.
In particular, Theorem~\ref{theorem:plane-curve} answers Question~\ref{question:plane-curves}.
This answer is very different from the answer given by Theorem~\ref{theorem:Ploski}.
Indeed, if $C_d$ is a P\l oski curve and $P$ is its singular point, then the formula in
Example~\ref{example:Kuwata} gives
$$
\mathrm{lct}_P\big(\mathbb{P}^2,C_d\big)=\frac{5}{2d}>\frac{2d-3}{(d-1)^2}.
$$

The proof of Theorem~\ref{theorem:plane-curve} implies one result
that is interesting on its own. To describe it, let us identify
the curve $C_d$ with a point in the space
$|\mathcal{O}_{\mathbb{P}^2}(d)|$ that parameterizes all (not
necessarily reduced) plane curves of degree $d$. Since the group
$\mathrm{PGL}_3(\mathbb{C})$ acts on
$|\mathcal{O}_{\mathbb{P}^2}(d)|$, it is natural to ask whether
$C_d$ is GIT-stable (resp., GIT-semistable) for this action or not. 
For small $d$, its answer is classical and immediately follows from the Hilbert--Mumford criterion (see \cite[Chapter~2.1]{MFK}).

\begin{example}[{\cite[Chapter~4.2]{MFK}}]
\label{example:Mumford} If $d=3$, then $C_3$ is GIT-stable (resp.,
GIT-semistable) if and only if $C_3$ is smooth (resp., $C_3$ has at most
$\mathbb{A}_1$ singularities). If $d=4$, then $C_4$ is GIT-stable
(resp., GIT-semistable) if and only if $C_4$ has at most
$\mathbb{A}_1$ and $\mathbb{A}_2$ singularities (resp., $C_4$ has at
most singular double points and $C_4$ is not a union of a cubic
with an inflectional tangent line).
\end{example}

Paul Hacking, Hosung Kim and Yongnam Lee noticed that the log
canonical threshold $\mathrm{lct}(\mathbb{P}^2,C_d)$ and
GIT-stability of the curve $C_d$ are closely related. In particular, they proved

\begin{theorem}[{\cite[Propositions~10.2~and~10.4]{Hacking}, \cite[Theorem~2.3]{KimLee}}]
\label{theorem:KimLee} If
$\mathrm{lct}(\mathbb{P}^2,C_d)\geqslant\frac{3}{d}$, then the
curve $C_d$ is GIT-semistable. If $d\geqslant 4$ and
$\mathrm{lct}(\mathbb{P}^2,C_d)>\frac{3}{d}$, then the curve $C_d$
is GIT-stable.
\end{theorem}

This gives a \emph{sufficient} condition for the curve $C_d$ to be
GIT-stable (resp, GIT-semistable). However, this condition is not
a \emph{necessary} condition. Let us give two examples that
illustrate this.

\begin{example}[{\cite[p.~268]{Wall1996}, \cite[Example~10.5]{Hacking}}]
\label{example:WahlHacking} Suppose that $d=5$, the quintic curve
$C_5$ is given by
$$
x^5+\Big(y^2-xz\Big)^2\Big(\frac{x}{4}+y+z\Big)=x^2\Big(y^2-xz\Big)\Big(x+2y\Big),
$$
and $P=[0:0:1]$. Then $C_5$ is irreducible and has singularity
$\mathbb{A}_{12}$ at the point $P$. In particular, it is rational.
Furthermore, the curve $C_5$ is GIT-stable
(see, for example, \cite[Chapter~4.2]{MFK}). On the other hand, it
follows from Example~\ref{example:Kuwata} that
$$
\mathrm{lct}\big(\mathbb{P}^2,C_5\big)=\mathrm{lct}_P\big(\mathbb{P}^2,C_5\big)=\frac{1}{2}+\frac{1}{13}=\frac{15}{26}<\frac{3}{5}.
$$
\end{example}

\begin{example}
\label{example:Ploski} Suppose that $C_d$ is a P\l oski curve. Let
$P$ be its singular point, and let $L$ be a general line in
$\mathbb{P}^2$. Then
$$
\mathrm{lct}\big(\mathbb{P}^2,C_{d}+L\big)=\mathrm{lct}\big(\mathbb{P}^2,C_d\big)=\mathrm{lct}_P\big(\mathbb{P}^2,C_d\big)=\frac{5}{2d}<\frac{3}{d}.
$$
On the other hand, if $d$ is even, then $C_d$ is
GIT-semistable, and $C_d+L$ is GIT-stable. This follows from the
Hilbert--Mumford criterion. Similarly, if $d$ is odd, then $C_d$
is GIT-unstable, and $C_d+L$ is GIT-semistable.
\end{example}

In this paper we will prove the following result that complements Theorem~\ref{theorem:KimLee}.

\begin{theorem}
\label{theorem:plane-curve-stability} If
$\mathrm{lct}(\mathbb{P}^2,C_d)<\frac{5}{2d}$, then $C_d$ is
GIT-unstable. Moreover, if
$\mathrm{lct}(\mathbb{P}^2,C_d)\leqslant\frac{5}{2d}$, then $C_d$
is not GIT-stable. Furthermore, if
$\mathrm{lct}(\mathbb{P}^2,C_d)=\frac{5}{2d}$, then $C_d$ is
GIT-semistable if and only if $C_d$ is an even P\l oski curve.
\end{theorem}

Example~\ref{example:Ploski} shows that this result is \emph{sharp}.
Surprisingly, its proof is very similar to the proof of Theorem~\ref{theorem:plane-curve}.
In fact, we will give a combined proof of both these theorems in Section~\ref{section:plane-curves}.

In this paper we will also prove one application of
Theorem~\ref{theorem:plane-curve}. To describe it, we need

\begin{definition}[{\cite[Appendix~A]{Tian2012}, \cite[Definition~1.20]{ChePark13}}]
\label{definition:alpha-function} For a given smooth variety $V$
equipped with an ample $\mathbb{Q}$-divisor $H_V$, let
$\alpha_{V}^{H_{V}}\colon V\to\mathbb{R}_{\geqslant 0}$ be a
function defined as
$$
\alpha_{V}^{H_{V}}(O)=\mathrm{sup}\left\{\lambda\in\mathbb{Q}\ \left|%
\aligned
&\text{the pair}\ \left(V, \lambda D_V\right)\ \text{is log canonical at $O$}\\
&\text{for every effective $\mathbb{Q}$-divisor}\ D_V\sim_{\mathbb{Q}} H_{V}\\
\endaligned\right.\right\}.%
$$
Denote its infimum by $\alpha(V,H_V)$.
\end{definition}

Let $S_d$ be a smooth surface in $\mathbb{P}^3$ of degree
$d\geqslant 3$, let $H_{S_d}$ be its hyperplane section, let $O$
be a point in $S_d$, and let $T_O$ be the hyperplane section of
$S_d$ that is singular at $O$. Similar to
$\mathrm{lct}_{P}(\mathbb{P}^2,C_d)$, we can define
$$
\mathrm{lct}_O\big(S_d,T_O\big)=\mathrm{sup}\Big\{\lambda\in\mathbb{Q}\ \Big|\ \text{the log pair}\ \big(S_d, \lambda T_O\big)\ \text{is log canonical at $O$}\Big\}.%
$$
Then $\alpha_{S_d}^{H_{S_d}}(O)\leqslant\mathrm{lct}_{O}(S_d,T_O)$ by Definition~\ref{definition:alpha-function}.
Note that $T_O$ is reduced, since the surface $S_d$ is smooth. In this paper we prove

\begin{theorem}
\label{theorem:main}  If
$\alpha_{S_d}^{H_{S_d}}(O)<\frac{2d-3}{d(d-2)}$, then
$$
\alpha_{S_d}^{H_{S_d}}(O)=\mathrm{lct}_{O}\big(S_d, T_O\big)\in\Bigg\{\frac{2}{d}, \frac{2d-3}{(d-1)^2}, \frac{2d-1}{d(d-1)}, \frac{2d-5}{d^2-3d+1}\Bigg\}.%
$$
Similarly, if $\alpha(S_d, H_{S_d})<\frac{2d-3}{d(d-2)}$, then
$$
\alpha\big(S_d, H_{S_d}\big)=\inf_{O\in S_d}\Big\{\mathrm{lct}_{O}\big(S_d, T_O\big)\Big\}\in\Bigg\{\frac{2}{d}, \frac{2d-3}{(d-1)^2}, \frac{2d-1}{d(d-1)}, \frac{2d-5}{d^2-3d+1}\Bigg\}.%
$$
\end{theorem}

If $d=3$, then we can drop the condition $\alpha_{S_d}^{H_{S_d}}(O)<\frac{2d-3}{d(d-2)}$ in Theorem~\ref{theorem:main}, since $\frac{2d-3}{d(d-2)}=1$ in this case.
Thus, Theorem~\ref{theorem:main} implies

\begin{corollary}[{\cite[Corollary~1.24]{ChePark13}}]
\label{corollary:cubic-surfaces} Suppose that $d=3$. Then $\alpha_{S_3}^{H_{S_3}}(O)=\mathrm{lct}_{O}(S_3, T_O)$.
\end{corollary}

If $d\geqslant 4$, we cannot drop the condition $\alpha_{S_d}^{H_{S_d}}(O)<\frac{2d-3}{d(d-2)}$ in Theorem~\ref{theorem:main} in general.
Let us give two examples that illustrate this.

\begin{example}
\label{example:lct-alpha-quartic} Suppose that $d=4$. Let $S_4$ be
a quartic surface in $\mathbb{P}^3$ that is given by
$$
t^3x+t^2yz+xyz(y+z)=0,
$$
and let $O$ be the point $[0:0:0:1]$. Then $S_4$ is smooth, and
$T_O$ has singularity $\mathbb{A}_1$ at $O$, which implies that
$\mathrm{lct}_{O}(S_4, T_O)=1$. Let $L_y$ be the line $x=y=0$, let
$L_z$ be the line $x=z=0$, and let $C_2$ be the conic
$y+z=xt+yz=0$. Then $L_y$, $L_z$ and $C_2$ are contained in $S_4$,
and $O=L_y\cap L_z\cap C_2$. Moreover,
$$
L_y+L_z+\frac{1}{2}C_2\sim 2H_{S_4},%
$$
because the divisor $2L_y+2L_z+C_2$ is cut out on $S_4$ by
$tx+yz=0$. Furthermore, the log pair
$(S_4,L_y+L_z+\frac{1}{2}C_2)$ is not log canonical at $O$, so that
$\alpha_{S_4}^{H_{S_4}}(O)<1$ by Definition~\ref{definition:alpha-function}.
\end{example}

\begin{example}
\label{example:lct-alpha} Suppose that $d\geqslant 5$ and $T_O$
has $\mathbb{A}_1$ singularity at $O$. Then $\mathrm{lct}_{O}(S_d,
T_O)=1$. Let $f\colon\widetilde{S}_d\to S_d$ be a blow up of the point
$O$. Denote by $E$ its exceptional curve. Then
$$
\Big(f^*(H_{S_d})-\frac{11}{5}E\Big)^2=5-\frac{121}{25}>0.
$$
Hence, it follows from Riemann--Roch theorem there is an integer
$n\geqslant 1$ such that the linear system $|f^*(5nH_{S_d})-11nE|$
is not empty. Pick a divisor $\widetilde{D}$ in this linear system,
and denote by $D$ its image on $S_d$. Then $(S_d,\frac{1}{5n}D)$
is not log canonical at $P$, since $\mathrm{mult}_{P}(D)\geqslant
11n$. On the other hand, $\frac{1}{5n}D\sim_{\mathbb{Q}} H_{S_d}$
by construction, so that $\alpha_{S_d}^{H_{_d}}(O)<1$ by Definition~\ref{definition:alpha-function}.
\end{example}

This work was was carried out during the author's stay at the Max Planck
Institute for Mathematics in Bonn in 2014. We would like to thank
the institute for the hospitality and very good working condition.
We would like to thank Michael Wemyss for checking the singularities of the curve $C_{5}$ in Example~\ref{example:WahlHacking}.
We would like to thank Alexandru Dimca, Yongnam Lee, Jihun Park, Hendrick S\"u\ss\ and Mikhail Zaidenberg for very useful comments.

\section{Preliminaries}
\label{section:preliminaries}

In this section, we present results that will be used in the proof of Theorems~\ref{theorem:plane-curve}, \ref{theorem:plane-curve-stability}, \ref{theorem:main}.
Let $S$ be a~smooth surface, let $D$ be an effective non-zero $\mathbb{Q}$-divisor on the surface $S$, and let $P$ be a point in the surface $S$.
Write
$$
D=\sum_{i=1}^{r}a_iC_i,
$$
where each $C_i$ is an irreducible curve on the surface $S$, and each $a_i$ is a non-negative rational number.
Let us recall

\begin{definition}[{\cite[\S~6]{CoKoSm}}]
\label{definition:lct-KLT} Let $\pi\colon\widetilde{S}\to S$ be a
birational morphism such that $\widetilde{S}$ is smooth. Then $\pi$ is
a composition of blow ups of smooth points.  For each $C_i$,
denote by $\widetilde{C}_i$ its proper transform on the surface
$\widetilde{S}$. Let $F_1,\ldots, F_n$ be $\pi$-exceptional curves.
Then
$$
K_{\widetilde{S}}+\sum_{i=1}^{r}a_i\widetilde{C}_i+\sum_{j=1}^{n}b_jF_j\sim_{\mathbb{Q}}\pi^{*}\big(K_{S}+D\big)
$$
for some rational numbers $b_1,\ldots,b_n$.
Suppose, in addition, that $\sum_{i=1}^r\widetilde{C}_i+\sum_{j=1}^n F_j$ is a divisor with simple normal crossings.
Then the log pair $(S,D)$ is said to be \emph{log canonical} at $P$ if and only if the following two conditions are satisfied:
\begin{itemize}
\item $a_i\leqslant 1$ for every $C_i$ such that $P\in C_i$,

\item $b_j\leqslant 1$ for every $F_j$ such that $\pi(F_j)=P$.
\end{itemize}
Similarly, the log pair $(S,D)$ is said to be \emph{Kawamata log
terminal} at $P$ if and only if $a_i<1$ for every $C_i$ such that
$P\in C_i$, and $b_j<1$ for every $F_j$ such that $\pi(F_j)=P$.
\end{definition}

Using just this definition, one can easily prove

\begin{lemma}
\label{lemma:three-curves} Suppose that $r=3$, $P\in C_1\cap
C_2\cap C_3$, the curves $C_1$, $C_2$ and $C_3$ are smooth at $P$,
$a_1<1$, $a_2<1$ and $a_3<1$. Moreover, suppose that both curves
$C_1$ and $C_2$ intersect the curve $C_3$ transversally at $P$.
Furthermore, suppose that $(S, D)$ is not Kawamata log terminal at
$P$. Put $k=\mathrm{mult}_{P}(C_1\cdot C_2)$. Then
$k(a_1+a_2)+a_3\geqslant k+1$.
\end{lemma}

\begin{proof}
Put $S_0=S$ and consider a sequence of blow ups
$$
\xymatrix{S_k\ar@{->}[rr]^{\pi_k}&& S_{k-1}\ar@{->}[rr]^{\pi_{k-1}}&& \cdots\ar@{->}[rr]^{\pi_3}&& S_2\ar@{->}[rr]^{\pi_2}&& S_1\ar@{->}[rr]^{\pi_1}&& S_0,}%
$$
where each $\pi_j$ is the blow up of the intersection point of the
proper transforms of the curves $C_1$ and $C_2$ on the surface
$S_{j-1}$ that dominates $P$ (such point exists, since
$k=\mathrm{mult}_{P}(C_1\cdot C_2)$). For each $\pi_j$, denote by
$E_j^k$ the proper transform of its exceptional curve on $S_k$.
For each $C_i$, denote by $C_i^k$ its proper transform on the
surface $S_{k}$. Then
$$
K_{S_{k}}+\sum_{i=1}^{n}a_iC_i^k+\sum_{j=1}^{k}\Big(j\big(a_1+a_2\big)+a_3-j\Big)E_j^k\sim_{\mathbb{Q}}(\pi_1\circ\pi_2\circ\cdots\circ\pi_k)^{*}\Big(K_{S}+D\Big),
$$
and $\sum_{i=1}^{n}C_i^k+\sum_{j=1}^{k}E_j$ is a simple normal
crossing divisor in every point of $\cup_{j=1}^{k}E_j$. Thus, it
follows from Definition~\ref{definition:lct-KLT} that there exists
$l\in\{1,\ldots,k\}$ such that $l(a_1+a_2)+a_3\geqslant l+1$,
because $(S, D)$ is not Kawamata log terminal at $P$. If $l=k$,
then we are done. So, we may assume that $l<k$. If
$k(a_1+a_2)+a_3<k+1$, then $a_1+a_2<1+\frac{1}{k}-a_3\frac{1}{k}$,
which implies that
$$
l+1\leqslant
l\big(a_1+a_2\big)+a_3<\Bigg(l+\frac{l}{k}-a_3\frac{l}{k}\Bigg)+a_3=l+\frac{l}{k}+a_3\Bigg(1-\frac{l}{k}\Bigg)\leqslant
l+\frac{l}{k}+\Bigg(1-\frac{l}{k}\Bigg)=l+1,
$$
because $a_3<1$. Thus, the obtained contradiction shows that $k(a_1+a_2)+a_3\geqslant k+1$.
\end{proof}

\begin{corollary}
\label{corollary:three-curves}
Suppose that $r=2$, $P\in C_1\cap C_2$, the curves $C_1$ and $C_2$ are smooth at $P$, $a_1<1$ and $a_2<1$.
Put $k=\mathrm{mult}_{P}(C_1\cdot C_2)$.
If $(S, D)$ is not Kawamata log terminal at $P$, then $k(a_1+a_2)\geqslant k+1$.
\end{corollary}

The log pair $(S,D)$ is called \emph{log canonical} if it is log canonical at every point of $S$.
Similarly, the log pair $(S,D)$ is called \emph{Kawamata log terminal} if it is Kawamata log terminal at every point of the surface $S$.

\begin{remark}
\label{remark:convexity}
Let $R$ be any effective $\mathbb{Q}$-divisor on $S$ such that $R\sim_{\mathbb{Q}} D$ and $R\ne D$.
Put
$$
D_{\epsilon}=(1+\epsilon) D-\epsilon R,
$$
where $\epsilon$ is a non-negative rational number.
Then $D_{\epsilon}\sim_{\mathbb{Q}} D$.
Moreover, since $R\ne D$, there exists the greatest rational number $\epsilon_0\geqslant 0$ such that the divisor $D_{\epsilon_0}$ is effective.
Then $\mathrm{Supp}(D_{\epsilon_0})$ does not contain at least one irreducible component of $\mathrm{Supp}(R)$.
Moreover, if $(S,D)$ is not log canonical at $P$, and $(S,R)$ is log canonical at $P$,
then $(S,D_{\epsilon_0})$ is not log canonical at $P$ by
Definition~\ref{definition:lct-KLT}, because
$$
D=\frac{1}{1+\epsilon_0}D_{\epsilon_0}+\frac{\epsilon_0}{1+\epsilon_0}
R
$$
and $\frac{1}{1+\epsilon_0}+\frac{\epsilon_0}{1+\epsilon_0}=1$.
Similarly, if the log pair $(S,D)$ is not Kawamata log terminal at
$P$, and $(S,R)$ is Kawamata log terminal at $P$, then
$(S,D_{\epsilon_0})$ is not Kawamata log terminal at $P$.
\end{remark}

The following result is well-known.

\begin{lemma}[{\cite[Exercise~6.18]{CoKoSm}}]
\label{lemma:Skoda} If $(S,D)$ is not log canonical at $P$, then
$\mathrm{mult}_{P}(D)>1$. Similarly, if $(S,D)$ is not Kawamata log terminal at $P$, then
$\mathrm{mult}_{P}(D)\geqslant 1$.
\end{lemma}

Combining with

\begin{lemma}[{\cite[Lemma~5.36]{CoKoSm}}]
\label{lemma:Pukhlikov} Suppose that $S$ is a smooth surface in
$\mathbb{P}^3$, and $D\sim_{\mathbb{Q}} H_{S}$, where $H_S$ is a
hyperplane section of $S$. Then each $a_i$ does not exceed $1$.
\end{lemma}

Lemma~\ref{lemma:Skoda} gives

\begin{corollary}
\label{corollary:Pukhlikov} Suppose that $S$ is a smooth surface
in $\mathbb{P}^3$, and $D\sim_{\mathbb{Q}} H_{S}$, where $H_S$ is
a hyperplane section of $S$. Then $(S,D)$ is log canonical outside
of finitely many points.
\end{corollary}

The following result is a special case of a much more general result, which is known as Shokurov's connectedness principle (see, for example, \cite[Theorem~6.3.2]{CoKoSm}).

\begin{lemma}[{\cite[Theorem~6.9]{Shokurov}}]
\label{lemma:Shokurov}
If $-(K_S+D)$ is big and nef, then the locus where $(S,D)$ is not Kawamata log terminal is connected.
\end{lemma}

\begin{corollary}
\label{corollary:plane-curves-connectedness} Let $C_d$ be a
reduced curve in $\mathbb{P}^2$ of degree $d$, let $O$ and $Q$ be
two points in $C_d$ such that $O\ne Q$. If
$\mathrm{lct}_{O}(\mathbb{P}^2,C_d)<\frac{3}{d}$, then
$\mathrm{lct}_{Q}(\mathbb{P}^2,C_d)\geqslant\frac{3}{d}$.
\end{corollary}

Let $\pi_1\colon S_1\to S$ be a~blow up of the point $P$, and let $E_1$ be the $\pi_1$-exceptional curve.
Denote by $D^1$ the~proper transform of the divisor $D$ on the surface $S_1$ via $\pi_1$.
Then the log pair $(S_1, D^1+(\mathrm{mult}_{P}(D)-1)E_1)$ is often called \emph{the log pull back} of the log pair $(S,D)$,
because
$$
K_{S_1}+D^1+\Big(\mathrm{mult}_{P}(D)-1\Big)E_1\sim_{\mathbb{Q}}\pi_1^{*}\big(K_{S}+D\big).
$$
This $\mathbb{Q}$-rational equivalence implies that the log pair $(S,D)$ is not log canonical at $P$ provided that $\mathrm{mult}_{P}(D)>2$.
Similarly, if $\mathrm{mult}_{P}(D)\geqslant 2$, then the singularities of the log pair $(S,D)$ are not Kawamata log terminal at the point $P$.

\begin{remark}
\label{remark:log-pull-back}
The log pair $(S,D)$ is log canonical at $P$ if and only if $(S_1, D^1+(\mathrm{mult}_{P}(D)-1)E_1)$ is log canonical at every point of the curve $E_1$.
Similarly, the log pair $(S,D)$ is Kawamata log terminal at $P$ if and only if $(S_1, D^1+(\mathrm{mult}_{P}(D)-1)E_1)$ is Kawamata log terminal at every point of the curve $E_1$.
\end{remark}

Let $Z$ be an irreducible curve on $S$ that contains $P$. Suppose
that $Z$ is smooth at $P$, and $Z$ is not contained in
$\mathrm{Supp}(D)$. Let $\mu$ be a non-negative rational number.
The following result is a very special case of a much more general
result known as \emph{Inversion of Adjunction} (see, for example,
\cite[\S~3.4]{Shokurov} or \cite[Theorem~6.29]{CoKoSm}).

\begin{theorem}[{\cite[Corollary~3.12]{Shokurov}, \cite[Exercise~6.31]{CoKoSm}, \cite[Theorem~7]{Ch13}}]
\label{theorem:adjunction} Suppose that the log pair $(S,\mu Z+D)$
is not log canonical at $P$ and $\mu\leqslant 1$. Then
$\mathrm{mult}_{P}(D\cdot Z)>1$.
\end{theorem}

This result implies

\begin{theorem}
\label{theorem:adjunction-2} Suppose that $(S,\mu Z+D)$ is not
Kawamata log terminal at $P$, and $\mu<1$. Then $\mathrm{mult}_{P}(D\cdot Z)>1$.
\end{theorem}

\begin{proof}
The log pair $(S, Z+D)$ is
not log canonical at $P$, because $\mu<1$, and $(S,\mu Z+D)$ is not Kawamata
log terminal at $P$. Then $\mathrm{mult}_{P}(D\cdot Z)>1$ by
Theorem~\ref{theorem:adjunction}.
\end{proof}

Theorems~\ref{theorem:adjunction} and \ref{theorem:adjunction-2}
imply

\begin{lemma}
\label{lemma:log-pull-back} If $(S,D)$ is not log canonical at $P$
and $\mathrm{mult}_{P}(D)\leqslant 2$, then there exists a
\emph{unique} point in $E_1$ such that $(S_1,
D^1+(\mathrm{mult}_{P}(D)-1)E_1)$ is not log canonical at it.
Similarly, if $(S,D)$ is not Kawamata log terminal at $P$,
and $\mathrm{mult}_{P}(D)<2$, then there exists a
\emph{unique} point in $E_1$ such that $(S_1,
D^1+(\mathrm{mult}_{P}(D)-1)E_1)$ is not Kawamata log terminal at
it.
\end{lemma}

\begin{proof}
If $\mathrm{mult}_{P}(D)\leqslant 2$ and $(S_1, D^1+(\lambda
\mathrm{mult}_{P}(D)-1)E_1)$ is not log canonical at two distinct
points $P_1$ and $\widetilde{P}_1$, then
$$
2\geqslant\mathrm{mult}_{P}\big(D\big)=D^1\cdot E_1\geqslant\mathrm{mult}_{P_1}\Big(D^1\cdot E_1\Big)+\mathrm{mult}_{\widetilde{P}_1}\Big(D^1\cdot E_1\Big)>2%
$$
by Theorem~\ref{theorem:adjunction}. By
Remark~\ref{remark:log-pull-back}, this proves the first
assertion. Similarly, we can prove the second assertion using
Theorem~\ref{theorem:adjunction-2} instead of
Theorem~\ref{theorem:adjunction}.
\end{proof}

The following result can be proved similarly to the proof of
Lemma~\ref{lemma:Skoda}. Let us show how to prove it using
Theorem~\ref{theorem:adjunction-2}.

\begin{lemma}
\label{lemma:Skoda-2} Suppose that $(S,D)$ is not Kawamata log
terminal at $P$, and $(S,D)$ is Kawamata log terminal in a
punctured neighborhood of the point $P$, then
$\mathrm{mult}_{P}(D)>1$.
\end{lemma}

\begin{proof}
By Remark~\ref{remark:log-pull-back}, the log pair $(S_1,
D^1+(\mathrm{mult}_{P}(D)-1)E_1)$ is not Kawamata log terminal at
some point $P_1\in E_1$. Moreover, if $\mathrm{mult}_{P}(D)<2$,
then $(S_1, D^1+(\mathrm{mult}_{P}(D)-1)E_1)$ is Kawamata log
terminal at a punctured neighborhood of the point $P_1$. Thus, if
$\mathrm{mult}_{P}(D)\leqslant 1$, then
$\mathrm{mult}_{P}\big(D\big)=D^1\cdot E_1>1$ by
Theorem~\ref{theorem:adjunction-2}, which is absurd.
\end{proof}

Let $Z_1$ and $Z_2$ be two irreducible curves on the surface $S$
such that $Z_1$ and $Z_2$ are not contained in $\mathrm{Supp}(D)$.
Suppose that $P\in Z_1\cap Z_2$, the curves $Z_1$ and $Z_2$ are
smooth at $P$, the curves $Z_1$ and $Z_2$ intersect each other
transversally at $P$. Let $\mu_1$ and $\mu_2$ be non-negative
rational numbers.

\begin{theorem}[{\cite[Theorem~13]{Ch13}}]
\label{theorem:Trento} Suppose that the log pair $(S,
\mu_1Z_1+\mu_2Z_2+D)$ is not log canonical at the point $P$, and
$\mathrm{mult}_{P}(D)\leqslant 1$. Then either
$\mathrm{mult}_{P}(D\cdot Z_{1})>2(1-\mu_{2})$ or
$\mathrm{mult}_{P}(D\cdot Z_{2})>2(1-\mu_{1})$ (or both).
\end{theorem}

This result implies

\begin{theorem}
\label{theorem:Trento-2} Suppose that $(S, \mu_1Z_1+\mu_2Z_2+D)$
is not Kawamata log terminal at $P$, and $\mathrm{mult}_{P}(D)<1$.
Then either $\mathrm{mult}_{P}(D\cdot Z_{1})\geqslant
2(1-\mu_{2})$ or $\mathrm{mult}_{P}(D\cdot Z_{2})\geqslant
2(1-\mu_{1})$ (or both).
\end{theorem}

\begin{proof}
Let $\lambda$ be a rational number such that
$$
\frac{1}{\mathrm{mult}_{P}(D)}\geqslant\lambda>1.
$$
Then
$(S,D+\lambda \mu_1Z_1+\lambda \mu_2Z_2)$ is not log canonical at
$P$. Now it follows from Theorem~\ref{theorem:Trento} that either
$\mathrm{mult}_{P}(D\cdot Z_{1})>2(1-\lambda \mu_{2})$ or
$\mathrm{mult}_{P}(D\cdot Z_{2})>2(1-\lambda \mu_{1})$ (or both).
Since we can choose $\lambda$ to be as close to $1$ as we wish, this
implies that either $\mathrm{mult}_{P}(D\cdot Z_{1})\geqslant
2(1-\mu_{2})$ or $\mathrm{mult}_{P}(D\cdot Z_{2})\geqslant
2(1-\mu_{1})$ (or both).
\end{proof}

\section{Reduced plane curves}
\label{section:plane-curves}

The purpose of this section is to prove
Theorems~\ref{theorem:plane-curve} and
\ref{theorem:plane-curve-stability}. Let $C_d$ be a \emph{reduced}
plane curve in $\mathbb{P}^2$ of degree $d\geqslant 4$, and let
$P$ be a point in $C_d$.  Put $\lambda_1=\frac{2d-3}{d(d-2)}$ and $\lambda_2=\frac{5}{2d}$.
To prove Theorem~\ref{theorem:plane-curve}, we have to show that if the log
pair $(\mathbb{P}^2,\lambda_1 C_d)$ is not Kawamata log terminal at the point $P$,
then one of the following assertions hold:
\begin{itemize}
\item $\mathrm{mult}_{P}(C_d)=d$,%
\item $C_d$ has singularity $\mathbb{T}_{d-1}$, $\mathbb{K}_{d-1}$, $\widetilde{\mathbb{T}}_{d-1}$ or $\widetilde{\mathbb{K}}_{d-1}$ at the point $P$, %
\item $d=4$ and $C_4$ is a P\l oski curve (see Definition~\ref{definition:Ploski-curve-even}).%
\end{itemize}
To prove Theorem~\ref{theorem:plane-curve-stability}, we have to
show that if $(\mathbb{P}^2,\lambda_2 C_d)$ is not Kawamata log
terminal, then either $C_d$ is GIT-unstable or $C_d$ is an even
P\l oski curve. In the rest of the section, we will
do this simultaneously. Let us start with few preliminary results.

\begin{lemma}
\label{lemma:plane-curve-inequalities} The following inequalities
hold:
\begin{enumerate}
\item[(i)] $\lambda_1<\frac{2}{d-1}$,

\item[(ii)] $\lambda_1<\frac{2k+1}{kd}$ for every positive integer $k\leqslant d-3$,%

\item[(iii)] if $d\geqslant 5$, then $\lambda_1<\frac{2k+1}{kd+1}$ for every positive integer $k\leqslant d-4$,%

\item[(iv)] $\lambda_1<\frac{3}{d}$,

\item[(v)] $\lambda_1<\frac{2}{d-2}$,

\item[(vi)] $\lambda_1<\frac{6}{3d-4}$,

\item[(vii)] if $d\geqslant 5$, then $\lambda_1<\lambda_2$.
\end{enumerate}
\end{lemma}

\begin{proof}
The equality $\frac{2}{d-1}=\lambda_1+\frac{d-3}{d(d-1)(d-2)}$
implies (i). Let $k$ be positive integer. If $k=d-2$, then
$\lambda_1=\frac{2k+1}{kd}$. This implies (ii), because
$\frac{2k+1}{kd}=\frac{2}{d}+\frac{1}{kd}$ is a decreasing
function on $k$ for $k\geqslant 1$. Similarly, if $k=d-4$ and
$d\geqslant 4$, then
$\lambda_1=\frac{2k+1}{kd+1}-\frac{3}{d(d-2)(d^2-4d+1)}<\frac{2k+1}{kd+1}$.
This implies (iii), since
$\frac{2k+1}{kd+1}=\frac{2}{d}+\frac{d-2}{d(kd+1)}$ is a
decreasing function on $k$ for $k\geqslant 1$. The equality
$\lambda_1=\frac{3}{d}-\frac{d-3}{d(d-2)}$ proves (iv). Note that
(v) follows from (i). Since $\frac{6}{3d-4}>\frac{2}{d-1}$, (vi)
also follows from (i). Finally, the equality
$\lambda_1=\lambda_2-\frac{d-4}{2d(d-2)}$ implies (vii).
\end{proof}

We may assume that $P=[0:0:1]$. Then $C_d$ is given by
$F_d(x,y,z)=0$, where $F_{d}(x,y,z)$ is a homogeneous polynomial
of degree $d$. Put $x_1=\frac{x}{z}$, $x_2=\frac{y}{z}$ and
$f_d(x_1,x_2)=F_d(x_1,x_2,1)$. Then
$$
f_d\big(x_1,x_2\big)=\sum_{\substack{i\geqslant 0, j\geqslant 0,\\ m_0\leqslant i+j\leqslant d}} \epsilon_{ij}x_1^ix_2^j,%
$$
where each $\epsilon_{ij}$ is a complex number. For every positive
integers $a$ and $b$, define the weight of the polynomial
$f_d(x_1,x_2)$ as
$$
\mathrm{wt}_{(a,b)}\big(f_d(x_1,x_2)\big)=\min\Big\{ai+bj\ \Big\vert\ \epsilon_{ij}\ne 0\Big\}.%
$$
Then the Hilbert--Mumford criterion implies

\begin{lemma}[{\cite[Lemma~2.1]{KimLee}}]%
\label{lemma:stability} Let $a$ and $b$ be positive integers. If
$C_d$ is GIT-stable, then
$$
\mathrm{wt}_{(a,b)}\Big(f_d\big(x_1,x_2\big)\Big)<\frac{d}{3}\big(a+b\big).
$$
Similarly,
if $C_d$ is GIT-semistable, then
$\mathrm{wt}_{(a,b)}(f_d(x_1,x_2))\leqslant\frac{d}{3}(a+b)$.
\end{lemma}

Let $f_1\colon S_1\to\mathbb{P}^2$ be a blow up of the point $P$.
Denote by $E_1$ the exceptional curve of the blow up $f_1$. Denote by $C^1_d$ the
proper transform on $S_1$ of the curve $C_d$.

\begin{lemma}
\label{lemma:stability-simple} If $\mathrm{mult}_{P}(C_d)>\frac{2d}{3}$, then $C_d$ is GIT-unstable.
Let $O$ be a point in $E_1$. If
$$
\mathrm{mult}_{P}(C_d)+\mathrm{mult}_{O}(C^1_d)>d,
$$
then $C_d$ is GIT-unstable.
\end{lemma}

\begin{proof}
Since $\mathrm{mult}_{P}(C_d)=\mathrm{wt}_{(1,1)}(f_d(x_1,x_2))$, the first assertion follows from Lemma~\ref{lemma:stability}.
Let us prove the second assertion.
We may assume that $O$ is contained in the proper transform of the line in $\mathbb{P}^2$ that is given by $x=0$.
Then
$$
\mathrm{wt}_{(2,1)}\big(f_d(x_1,x_2)\big)=\mathrm{mult}_{P}(C_d)+\mathrm{mult}_{O}(C^1_d),
$$
so that the second assertion also follows from Lemma~\ref{lemma:stability}.
\end{proof}

Now we are ready to prove Theorems~\ref{theorem:plane-curve} and \ref{theorem:plane-curve-stability}.
To do this, we may assume that $C_d$ is not a union of $d$ lines passing through the point $P$.
Suppose, in addition, that
\begin{itemize}
\item[(\textbf{A})] either $(\mathbb{P}^2,\lambda_1 C_d)$ is not Kawamata log terminal at $P$,%
\item[(\textbf{B})] or $(\mathbb{P}^2,\lambda_2 C_d)$ is not Kawamata log terminal at $P$.%
\end{itemize}
We will show that (\textbf{A}) implies that either $C_d$ has singularity $\mathbb{T}_{d-1}$, $\mathbb{K}_{d-1}$, $\widetilde{\mathbb{T}}_{d-1}$ or $\widetilde{\mathbb{K}}_{d-1}$ at the point $P$, or $C_d$ is a P\l oski quartic curve.
Similarly, we will show that (\textbf{B}) implies that either $C_d$ is GIT-unstable (i.e. $C_d$ is not GIT-semistable), or $C_d$ is an even P\l oski curve.
If (\textbf{A}) holds, let $\lambda=\lambda_1$. If (\textbf{B}) holds, let $\lambda=\lambda_2$.

If $d=4$, then $\lambda_1=\lambda_2$.
If $d\geqslant 5$, then $\lambda_1<\lambda_2$ by Lemma~\ref{lemma:plane-curve-inequalities}(vii).
Since $C_d$ is reduced and $\lambda<1$, the log pair $(\mathbb{P}^2,\lambda C_d)$ is Kawamata log terminal outside of finitely many points.
Thus, it is Kawamata log terminal outside of $P$ by Lemma~\ref{lemma:Shokurov}.

Put $m_0=\mathrm{mult}_P(C_d)$.
Then the log pair $(S_1,\lambda C^1_d+(\lambda m_0-1)E_1)$ is not Kawamata log terminal at some point $P_1\in E_1$ by Remark~\ref{remark:log-pull-back}.
Note that we have
$$
K_{S_1}+\lambda C^1_d+\Big(\lambda
m_0-1\Big)E_1\sim_{\mathbb{Q}}f_1^*\Big(K_{\mathbb{P}^2}+\lambda C_d\Big).%
$$
Let $f_2\colon S_2\to S_1$ be a blow up of the point $P_1$, and let $E_2$ be its exceptional curve. Denote by $C^2_d$ the proper
transform on $S_2$ of the curve $C_d$, and denote by $E_1^2$ the
proper transform on $S_2$ of the curve $E_1$. Put $m_1=\mathrm{mult}_{P_1}(C_d^1)$. Then
$$
K_{S_2}+\lambda C^2_d+\big(\lambda m_0-1\big)E_1^2+\big(\lambda(m_0+m_1)-2\big)E_2\sim_{\mathbb{Q}} f_2^*\Big(K_{S_1}+\lambda C^1_d+\big(\lambda m_0-1\big)E_1\Big).%
$$
By Remark~\ref{remark:log-pull-back}, the log pair $(S_2,\lambda C^2_d+(\lambda m_0-1)E_1^2+(\lambda (m_0+m_1)-2)E_2)$ is not Kawamata log terminal at some point $P_2\in E_2$.
Let $f_3\colon S_3\to S_2$ be a blow up of this point, and let $E_3$ be the $f_3$-exceptional curve.
Denote by $C^3_d$ the proper transform on $S_3$ of the curve $C_d$,
denote by $E_1^3$ the proper transform on $S_3$ of the curve $E_1$, and denote by $E_2^3$ the proper transform on $S_3$ of the curve $E_2$.
Put $m_2=\mathrm{mult}_{P_2}(C_d^2)$. Then
\begin{multline*}
K_{S_3}+\lambda_2 C^3_d+\big(\lambda_2 m_0-1\big)E_1^3+\\
+\big(\lambda_2(m_0+m_1)-2\big)E_2^3+\big(\lambda_2(2m_0+m_1+m_2)-4\big)E_3\sim_{\mathbb{Q}}\\
\sim_{\mathbb{Q}} f_3^*\Big(K_{S_2}+\lambda_2 C^2_d+\big(\lambda_2 m_0-1\big)E_1^2+\big(\lambda_2(m_0+m_1)-2\big)E_2\Big).%
\end{multline*}
Thus, the log pair $(S_3,\lambda_2 C^3_d+(\lambda_2 m_0-1)E_1^3+(\lambda_2(m_0+m_1)-2)E_2^3+(\lambda_2(2m_0+m_1+m_2)-4)E_3)$ is not Kawamata log terminal at some point $P_3\in E_3$ by Remark~\ref{remark:log-pull-back}.
Note that the divisor $\lambda_2 C^3_d+(\lambda_2 m_0-1)E_1^3+(\lambda_2(m_0+m_1)-2)E_2^3+(\lambda_2(2m_0+m_1+m_2)-4)E_3$
is effective by Lemma~\ref{lemma:Skoda}.

\begin{lemma}
\label{lemma:plane-curve-mult} One has $\lambda m_0<2$.
\end{lemma}

\begin{proof}
Since $C_d$ is not a union of $d$ lines passing through $P$, we have $m_0\leqslant d-1$.
Thus, if (\textbf{A}) holds, then $\lambda m_0<2$ by Lemma~\ref{lemma:plane-curve-inequalities}(i), because $d\geqslant 4$.
Similarly, if (\textbf{B}) holds, then $m_0\leqslant\frac{2d}{3}$ by Lemma~\ref{lemma:stability-simple}, which implies that $\lambda m_0\leqslant\frac{10}{6}<2$.
\end{proof}

Thus, the log pair $(S_1,\lambda C^1_d+(\lambda m_0-1)E_1)$ is
Kawamata log terminal outside of $P_1$ by
Lemma~\ref{lemma:log-pull-back}. Note that $P_1\in C^1_d$,
because the log pair $(S_1, (\lambda m_0-1)E_1)$ is
not Kawamata log terminal at $P_1$. Thus, we have
$m_1>0$.

Let $L$ be the line in $\mathbb{P}^2$ whose proper transform on $S_1$ contains the point $P_1$.
Such a line exists and it is unique. By a suitable linear change of coordinates, we may assume that $L$ is given by $x=0$.
Denote by $L^1$ the proper transform of the line $L$ on the surface $S_1$.

\begin{lemma}
\label{lemma:plane-curve-cusp} Suppose that (\textbf{A}) holds and $m_0=d-1$.
Then $C_d$ has singularity $\mathbb{K}_{d-1}$,
$\widetilde{\mathbb{K}}_{d-1}$, $\mathbb{T}_{d-1}$ or
$\widetilde{\mathbb{T}}_{d-1}$ at the point $P$.
\end{lemma}

\begin{proof}
Suppose that $L$ is not an irreducible component of the curve
$C_d$. Then $m_0+m_1\leqslant d$, because
$$
d-1-m_0=C_d^1\cdot L^1\geqslant m_1.
$$
Since $m_0=d-1$, this gives $m_1=1$. Then $P_1\in C_{d}^1$ and
the curve $C_{d}^1$ is smooth at $P_1$. Put
$k=\mathrm{mult}_{P_1}(C_d^1\cdot E_1)$. Applying
Corollary~\ref{corollary:three-curves} to the log pair
$(S_1,\lambda_1 C^1_d+(\lambda_1 m_0-1)E_1)$ at the point $P_1$, we get
$$
k\lambda_1 m_0\geqslant k+1,
$$
which gives $\lambda_1\geqslant\frac{2k+1}{kd}$. Then $k\geqslant
d-2$ by Lemma~\ref{lemma:plane-curve-inequalities}(ii). Since
$$
k\leqslant C_d^1\cdot E_1=m_0=d-1,
$$
either $k=d-1$ or $k=d-2$. If $k=d-1$, then $C_d$ has singularity
$\mathbb{K}_{d-1}$ at $P$. If $k=d-2$, then $C_d$ has singularity
$\widetilde{\mathbb{K}}_{d-1}$ at the point $P$.

To complete the proof, we may assume that $L$ is an
irreducible component of the curve $C_d$. Then $C_{d}=L+C_{d-1}$,
where $C_{d-1}$ is a reduced curve in $\mathbb{P}^2$ of degree
$d-1$ such that $L$ is not its irreducible component. Denote by
$C_{d-1}^1$ its proper transform on $S_1$. Put
\hbox{$n_0=\mathrm{mult}_{P}(C_{d-1})$} and
$n_1=\mathrm{mult}_{P_1}(C_{d-1}^1)$. Then $n_0=m_0-1=d-2$ and
$n_1=m_1-1$. This implies that $P_1\in
C_{d-1}^1$, since the log pair $(S_1,\lambda_1 L^1+(\lambda_1
m_0-1)E_1)$ is Kawamata log terminal at $P$.
Hence, $n_1\geqslant 1$. One the other hand, we have
$$
d-1-n_0=C_{d-1}^1\cdot L^1\geqslant  n_1,
$$
which implies that $n_0+n_1\leqslant d-1$. Then $n_1=1$, since
$n_0=d-2$.

We have $P_1\in C_{d-1}^1$ and $C_{d-1}^1$ is smooth at $P_1$.
Moreover, since
$$
1=d-1-n_0=L^1\cdot C_{d-1}^1\geqslant  n_1=1,
$$
the curve $C_{d-1}^1$ intersects the curve $L^1$ transversally at
the point $P_1$. Put $k=\mathrm{mult}_{P_1}(C_{d-1}^1\cdot E_1)$.
Then $k\geqslant 1$. Applying Lemma~\ref{lemma:three-curves} to
the log pair $(S_1,\lambda_1 C^1_{d-1}+\lambda_1 L^1+(\lambda_1
(n_0+1)-1)E_1)$ at the point $P_1$, we get
$$
k\Big(\lambda_1(n_0+2)-1\Big)+\lambda_1\geqslant k+1.
$$
Then $\lambda_1\geqslant\frac{2k+1}{kd+1}$. Then $k\geqslant d-3$
by Lemma~\ref{lemma:plane-curve-inequalities}(iii). Since
$$
k\leqslant E_1\cdot C_{d-1}^1=n_0=d-2,
$$
either $k=d-2$ or
$k=d-3$. In the former case, $C_d$ has singularity $\mathbb{T}_{d-1}$ at the point $P$.
In the latter case, $C_d$ has singularity $\widetilde{\mathbb{T}}_{d-1}$ at the point $P$.
\end{proof}

\begin{lemma}
\label{lemma:plane-curve-line-Supp} Suppose that (\textbf{A}) holds and $m_0\leqslant d-2$.
Then the line $L$ is not an irreducible component of the curve $C_d$.
\end{lemma}

\begin{proof}
Suppose that $L$ is an irreducible
component of the curve $C_d$. Let us see for a contradiction. Put
$C_{d}=L+C_{d-1}$, where $C_{d-1}$ is a reduced curve in
$\mathbb{P}^2$ of degree $d-1$ such that $L$ is not its
irreducible component. Denote by $C_{d-1}^1$ its proper transform
on $S_1$. Put $n_0=\mathrm{mult}_{P}(C_{d-1})$ and
$n_1=\mathrm{mult}_{P_1}(C_{d-1}^1)$. Then $(S_1,
(\lambda_1(n_0+1)-1)E_1+\lambda_1 L^1+\lambda_1 C_{d-1}^1)$ is not
Kawamata log terminal at $P_1$ and is Kawamata log terminal
outside of the point $P_1$. In particular, $n_1\ne 0$, because
$(S_1, (\lambda_1(n_0+1)-1)E_1+\lambda_1 L^1)$ is Kawamata log
terminal at $P_1$. On the other hand,
$$
d-1-n_0=L^1\cdot C_{d-1}^1\geqslant n_1,
$$
which implies that $n_0+n_1\leqslant d-1$. Furthermore, we have
$n_0=m_0-1\leqslant d-3$.

Since $n_0+n_1\geqslant 2n_1$, we have $n_1\leqslant
\frac{d-1}{2}$. Then $\lambda n_1<1$ by
Lemma~\ref{lemma:plane-curve-inequalities}(i). Thus, we can apply
Theorem~\ref{theorem:Trento-2} to the log pair $(S_1,
(\lambda_1(n_0+1)-1)E_1+\lambda_1 L^1+\lambda_1 C_{d-1}^1)$ at the
point $P_1$.  This gives either
$$
\lambda_1\big(d-1-n_0\big)=\lambda_1 C_{d-1}^1\cdot L^1\geqslant
2\Big(2-\lambda_1\big(n_0+1\big)\Big)
$$
or
$$
\lambda_1 n_0=\lambda_1 C_{d-1}^1\cdot E_1\geqslant 2\big(1-\lambda_1\big)%
$$
(or both). In the former case, we have $\lambda_1(d+1+n_0)\geqslant
4$. In the latter case, we have \hbox{$\lambda_1(n_0+2)>2$}.
Thus, in both cases we have $\lambda_1(d-1)\geqslant 2$, since $n_0\leqslant d-3$.
But $\lambda_1(d-1)<2$ by Lemma~\ref{lemma:plane-curve-inequalities}(i).
This is a contradiction.
\end{proof}

If the curve $C_d$ is GIT-semistable, then $m_0\leqslant d-2$ by Lemma~\ref{lemma:stability-simple}.
Thus, it follows from Lemma~\ref{lemma:plane-curve-cusp} that we may assume that
$$
m_0\leqslant d-2
$$
in order to complete the proof of Theorems~\ref{theorem:plane-curve} and \ref{theorem:plane-curve-stability}.
Moreover, if $L$ is not an irreducible component of the curve $C_d$, then
$$
d-m_0=C_d^1\cdot L^1\geqslant m_1.
$$
Thus, if (\textbf{A}) holds, then $m_0+m_1\leqslant d$ by Lemma~\ref{lemma:plane-curve-line-Supp}.
Similarly, if the curve $C_d$ is GIT-semistable, then $m_0+m_1\leqslant d$  by Lemma~\ref{lemma:stability-simple}.
Thus, to complete the proof of Theorems~\ref{theorem:plane-curve} and \ref{theorem:plane-curve-stability}, we may also assume that
\begin{equation}
\label{equation:plane-curve-m0-m1-d}
m_0+m_1\leqslant d.
\end{equation}
Then $\lambda (m_0+m_1)<3$ by Lemma~\ref{lemma:plane-curve-inequalities}(v),
so that $(S_2,\lambda C^2_d+(\lambda m_0-1)E_1^2+(\lambda (m_0+m_1)-2)E_2)$
is Kawamata log terminal outside of the point $P_2$ by Lemma~\ref{lemma:log-pull-back}.
Furthermore, we have

\begin{lemma}
\label{lemma:plane-curve-P2-E1-E2}
Suppose that $P_2=E_1^2\cap E_2$.
Then (\textbf{A}) does not hold and $C_d$ is GIT-unstable.
\end{lemma}

\begin{proof}
We have $m_0-m_1=E_1^2\cdot C_d^2\geqslant m_2$, so that
\begin{equation}
\label{equation:new-equation}
m_2\leqslant \frac{m_0}{2},
\end{equation}
because $2m_2\leqslant m_1+m_2$.
On the other hand, $m_0\leqslant d-2$ by assumption.
Thus, we have $m_2\leqslant\frac{d-2}{2}$.

Suppose that (\textbf{A}) holds. Then $\lambda=\lambda_1$ and $\lambda_1
m_2<1$ by Lemma~\ref{lemma:plane-curve-inequalities}(v). Thus, we can apply Theorem~\ref{theorem:Trento-2} to the log pair
$(S_2,\lambda_1 C^2_d+(\lambda_1 m_0-1)E_1^2+(\lambda_1
(m_0+m_1)-2)E_2)$. This gives either
$$
\lambda_1\big(m_0-m_1\big)=\lambda_1 C_{d}^2\cdot E_1^2\geqslant 2\Big(3-\lambda_1\big(m_0+m_1\big)\Big)%
$$
or
$$
\lambda_1 m_1=\lambda_1 C_{d}^2\cdot E_2\geqslant 2\big(2-\lambda_1 m_0\big)%
$$
(or both). The former inequality implies
$\lambda_1(3m_0+m_1)\geqslant 6$. The latter inequality implies
$\lambda_1(2m_0+m_1)\geqslant 4$. On the other hand,
$m_0+m_1\leqslant d$ by \eqref{equation:plane-curve-m0-m1-d}, and
$m_0\leqslant d-2$ by assumption.  Thus, $3m_0+m_1\leqslant 3d-4$
and $2m_0+m_1\leqslant 2d-2$. Then $\lambda_1(3m_0+m_1)<6$ by
Lemma~\ref{lemma:plane-curve-inequalities}(vi), and
$\lambda_1(2m_0+m_1)<4$ by
Lemma~\ref{lemma:plane-curve-inequalities}(i).
The obtained contradiction shows that (\textbf{A}) does not hold.

We see that (\textbf{B}) holds. We have to show that $C_d$ is GIT-unstable.
Suppose that this is not the case, so that $C_d$ is GIT-semistable.
Let us seek for a contradiction.

By Lemma~\ref{lemma:stability}, we have $2m_0+m_1+m_2\leqslant \frac{5d}{3}$, because
$$
\mathrm{wt}_{(3,2)}\Big(f_d\big(x_1,x_2\big)\Big)=2m_0+m_1+m_2.
$$
Thus, we have $\lambda_2(2m_0+m_1+m_2)-4<1$ by Lemma~\ref{lemma:plane-curve-inequalities}(v).
Hence, the log pair $(S_3,\lambda_2 C^3_d+(\lambda_2 m_0-1)E_1^3+(\lambda_2(m_0+m_1)-2)E_2^3+(\lambda_2(2m_0+m_1+m_2)-4)E_3)$
is Kawamata log terminal outside of the point $P_3$ by Remark~\ref{remark:log-pull-back}.

If $P_3=E_1^3\cap E_3$, then it follows from Theorem~\ref{theorem:adjunction-2} that
$$
\lambda_2\big(m_0-m_1-m_2\big)=\lambda_2 C_d^3\cdot E_1^3>5-\lambda_2\big(2m_0+m_1+m_2\big),%
$$
which implies that $m_0>\frac{5}{3\lambda_2}=\frac{2d}{3}$, which is impossible by Lemma~\ref{lemma:stability-simple}.
If $P_3=E_2^3\cap E_3$, then it follows from Theorem~\ref{theorem:adjunction-2} that
$$
\lambda_2\big(m_1-m_2\big)=\lambda_2C_d^3\cdot E_2^3>5-\lambda_2\big(2m_0+m_1+m_2\big),%
$$
which implies that $m_0+m_1>\frac{5}{2\lambda_2}=d$, which is impossible by Lemma~\ref{lemma:stability-simple}.
Thus, we see that $P_3\not\in E_1^3\cup E_2^3$.
Then the log pair $(S_3,\lambda_2 C^3_d+(\lambda_2(2m_0+m_1+m_2)-4)E_3)$ is not Kawamata log terminal at $P_3$.
Hence, Theorem~\ref{theorem:adjunction-2} gives
$$
\lambda_2 m_2=\lambda_2 C_d^3\cdot E_3>1,
$$
which implies that $m_2>\frac{1}{\lambda_2}=\frac{2d}{5}$.
Then $m_0>\frac{4d}{5}$ by \eqref{equation:new-equation},
which is impossible by Lemma~\ref{lemma:stability-simple}.
\end{proof}

Thus, to complete the proof of Theorems~\ref{theorem:plane-curve} and \ref{theorem:plane-curve-stability},
we may assume that
$$
P_2\ne E_1^2\cap E_2.
$$
Denote by $L^2$ the proper transform of the line $L$ on the surface $S_2$.

\begin{lemma}
\label{lemma:plane-curve-P2-L-E2} One has $P_2\ne L^2\cap E_2$.
\end{lemma}

\begin{proof}
Suppose that $P_2=L^2\cap E_2$. If $L$ is not an irreducible
component of the curve $C_d$, then
$$
d-m_0-m_1=L^2\cdot E_2\geqslant m_2,
$$
which implies that $m_0+m_1+m_2\leqslant d$. Thus, if (\textbf{A})
holds, then $\lambda=\lambda_1$ and $L$ is not an irreducible
component of the curve $C_d$ by
Lemma~\ref{lemma:plane-curve-line-Supp}, which implies that
$$
\lambda_1 d\geqslant \lambda_1\big(m_0+m_1+m_2\big)>3
$$
by Lemma~\ref{lemma:Skoda-2}. On the other hand, $\lambda_1 d<3$
by Lemma~\ref{lemma:plane-curve-inequalities}(iv). This shows that
(\textbf{B}) holds.

Since $\lambda=\lambda_2=\frac{5}{2d}<\frac{3}{d}$ and
$\lambda_2(m_0+m_1+m_2)>3$ by Lemma~\ref{lemma:Skoda-2}, we have
$m_0+m_1+m_2>d$. In particular, the line $L$ must be an
irreducible component of the curve $C_d$.

Put $C_{d}=L+C_{d-1}$, where $C_{d-1}$ is a reduced curve in
$\mathbb{P}^2$ of degree $d-1$ such that $L$ is not its
irreducible component. Denote by $C_{d-1}^1$ its proper transform
on $S_1$, and denote by $C_{d-1}^2$ its proper transform on $S_2$.
Put $n_0=\mathrm{mult}_{P}(C_{d-1})$,
$n_1=\mathrm{mult}_{P_1}(C_{d-1}^1)$ and
$n_2=\mathrm{mult}_{P_2}(C_{d-1}^2)$. Then $(S_2,
(\lambda_2(n_0+n_1+2)-2)E_2+\lambda_2 L^1+\lambda_2 C_{d-1}^1)$ is
not Kawamata log terminal at $P_2$ and is Kawamata log terminal
outside of the point $P_2$. Then
Theorem~\ref{theorem:adjunction-2} implies
$$
\lambda_2\big(d-1-n_0-n_1\big)=\lambda_2 C_{d-1}^2\cdot L^2>1-\big(\lambda_2(n_0+n_1+2)-2\big)=3-\lambda_2(n_0+n_1+2),%
$$
which implies that $\frac{5(d+1)}{2d}=\lambda_2(d+1)>3$. Hence,
$d=4$. Then $\lambda=\lambda_2=\frac{5}{8}$.

By \eqref{equation:plane-curve-m0-m1-d},
$n_0+n_1\leqslant 2$. Thus, $n_0=n_1=n_2=1$, since
$$
\frac{5}{8}\big(n_0+n_1+n_2+3\big)=\lambda_2\big(m_0+m_1+m_2\big)>3
$$
by Lemma~\ref{lemma:Skoda-2}. Then $C_{3}$ is a irreducible cubic
curve that is smooth at $P$, the line $L$ is tangent to the curve
$C_3$ at the point $P$, and $P$ is an inflexion point of the cubic
curve $C_3$. This implies that $\mathrm{lct}_P(\mathbb{P}^2,
C_d)=\frac{2}{3}$. Since $\frac{2}{3}>\frac{5}{8}=\lambda_2$, the
log pair $(\mathbb{P}^2, \lambda_2 C_d)$ must be Kawamata log
terminal at the point $P$, which contradicts (\textbf{B}).
\end{proof}

Recall that $m_0+m_1\leqslant d$ by \eqref{equation:plane-curve-m0-m1-d}.
Then $m_1\leqslant\frac{d}{2}$, since $2m_1\leqslant m_0+m_1$.
Thus, we have
\begin{equation}
\label{equation:plane-curve-mult-S3}
\lambda\big(m_0+m_1+m_2\big)\leqslant\lambda\big(m_0+2m_1\big)\leqslant\lambda\frac{3d}{2}\leqslant\lambda_2\frac{3d}{2}=\frac{15}{4}<4.%
\end{equation}
Therefore, the log pair $(S_3,\lambda C^3_d+(\lambda(m_0+m_1)-2)E_2^3+(\lambda(m_0+m_1+m_2)-3)E_3)$ is Kawamata log terminal outside of the point $P_3$ by Lemma~\ref{lemma:log-pull-back}.

\begin{lemma}
\label{lemma:plane-curve-P3-E2-E3} One has $P_3\ne E_2^3\cap E_3$.
\end{lemma}

\begin{proof}
If $P_3=E_2^3\cap E_3$, then Theorem~\ref{theorem:adjunction-2}
gives
$$
\lambda\big(m_1-m_2\big)=\lambda C_d^3\cdot E_2^3>1-\Big(\lambda\big(m_0+m_1+m_2\big)-3\Big)=4-\lambda\big(m_0+m_1+m_2\big),%
$$
which implies that $\lambda(m_0+2m_1)>4$. But $\lambda
(m_0+2m_1)<4$ by \eqref{equation:plane-curve-mult-S3}.
\end{proof}

Let $f_4\colon S_4\to S_3$ be a blow up of the point $P_3$, and
let $E_4$ be its exceptional curve. Denote by $C^4_d$ the proper
transform on $S_4$ of the curve $C_d$, denote by $E_3^4$ the
proper transform on $S_4$ of the curve $E_3$, and denote by $L^4$
the proper transform of the line $L$ on the surface $S_4$. Then
$(S_4,\lambda
C^4_d+(\lambda(m_0+m_1+m_2)-3)E_3^4+(\lambda(m_0+m_1+m_2+m_3)-4)E_4)$
is not Kawamata log terminal at some point $P_4\in E_4$ by
Remark~\ref{remark:log-pull-back}.
Moreover, we have
$$
2L^4+E_1+2E_2+E_3\sim (f_1\circ f_2\circ f_3\circ f_4)^*\Big(\mathcal{O}_{\mathbb{P}^2}\big(2\big)\Big)-(f_2\circ f_3\circ f_4)^*\big(E_1\big)-(f_3\circ f_4)^*\big(E_2\big)-f_4^*\big(E_3\big)-E_4.%
$$

\begin{lemma}
\label{lemma:plane-curve-pencil} The linear system
$|2L^4+E_1+2E_2+E_3|$ is a pencil that does not have base points.
Moreover, every divisor in $|2L^4+E_1+2E_2+E_3|$ that is different
from $2L^4+E_1+2E_2+E_3$ is a  smooth curve whose image on
$\mathbb{P}^2$ is a smooth conic that is tangent to $L$ at the
point $P$.
\end{lemma}

\begin{proof}
All assertions follows from $P_2\not\in E_1^2\cup L^2$ and
$P_3\not\in E_2^3$.
\end{proof}

Let $C_2^4$ be a general curve in $|2L^4+E_1+2E_2+E_3|$. Denote by
$C_2$ its image on $\mathbb{P}^2$, and denote by $\mathcal{L}$ the
pencil generated by $2L$ and $C_2$. Then $P$ is the only base
point of the pencil $\mathcal{L}$, and every conic in
$\mathcal{L}$ except $2L$ and $C_2$ intersects $C_2$ at $P$ with
multiplicity $4$ (cf. \cite[Remark~1.14]{ChePark13}).

\begin{lemma}
\label{lemma:plane-curve-mult-S4} One has
$m_0+m_1+m_2+m_3\leqslant m_0+m_1+2m_2\leqslant\frac{5}{\lambda}$.
If $m_0+m_1+m_2+m_3=\frac{5}{\lambda}$, then $d$ is even and $C_d$
is a union of $\frac{d}{2}\geqslant 2$ smooth conics in
$\mathcal{L}$, where $d=4$ if (\textbf{A}) holds.
\end{lemma}

\begin{proof}
By \eqref{equation:plane-curve-m0-m1-d}, we have $m_2+m_3\leqslant
2m_2\leqslant m_0+m_1\leqslant d$. This gives
$$
m_0+m_1+m_2+m_3\leqslant m_0+m_1+2m_2\leqslant 2d=\frac{5}{\lambda_2}\leqslant\frac{5}{\lambda}.%
$$

To complete the proof, we may assume that
$m_0+m_1+m_2+m_3=\frac{5}{\lambda}$. Then all inequalities above
must be equalities. Thus, we have $m_2=m_3=\frac{d}{2}$ and
$\lambda_1=\lambda_2$. In particular, if (\textbf{A}) holds, then
$d=4$, because $\lambda_1<\lambda_2=\frac{5}{2d}$ for $d\geqslant
5$ by Lemma~\ref{lemma:plane-curve-inequalities}(vii). Moreover,
since $m_0\geqslant m_1\geqslant m_2=\frac{d}{2}$ and
$m_0+m_1\leqslant d$, we see that $m_0=m_1=\frac{d}{2}$. Thus, $d$
is even and
$$
C_d^4\sim \frac{d}{2}\Big(2L^4+E_1+2E_2+E_3\Big),
$$
where $d=4$ if (\textbf{A}) holds. Since $|2L^4+E_1+2E_2+E_3|$ is
a free pencil and $C_d^4$ is reduced, it follows from
Lemma~\ref{lemma:plane-curve-pencil} that $C_d^4$ is a union of
$\frac{d}{2}$ smooth curves in $|2L^4+E_1+2E_2+E_3|$. In
particular, $L^4$ is not an irreducible component of $C_d^4$.
Thus, the curve $C_d$ is a union of $\frac{d}{2}$ smooth conics in
$\mathcal{L}$, where $d=4$ if (\textbf{A}) holds.
\end{proof}

We see that $m_0+m_1+m_2+m_3\leqslant\frac{5}{\lambda}$.
Moreover, if $m_0+m_1+m_2+m_3=\frac{5}{\lambda}$, then $C_d$ is an even P\l oski curve.
Furthermore, if $m_0+m_1+m_2+m_3=\frac{5}{\lambda}$ and (\textbf{A}) holds, then $d=4$.
Thus, to prove Theorems~\ref{theorem:plane-curve} and \ref{theorem:plane-curve-stability}, we may assume that
$$
m_0+m_1+m_2+m_3<\frac{5}{\lambda}.
$$
Let us show that this assumption leads to a contradiction.
By Lemma~\ref{lemma:log-pull-back},  this inequality implies that the log pair $(S_4,\lambda C^4_d+(\lambda(m_0+m_1+m_2)-3)E_3^4+(\lambda(m_0+m_1+m_2+m_3)-4)E_4)$ is Kawamata log terminal outside of the point $P_4$.

\begin{lemma}
\label{lemma:plane-curve-P4-E3-E4} One has $P_4\ne E_3^4\cap E_4$.
\end{lemma}

\begin{proof}
By Lemma~\ref{lemma:plane-curve-mult-S4},  $m_0+m_1+2m_2\leqslant
\frac{5}{\lambda}$. If $P_4=E_3^4\cap E_4$, then
Theorem~\ref{theorem:adjunction-2} gives
$$
\lambda\big(m_2-m_3\big)=\lambda C_d^4\cdot E_3^4>5-\lambda\big(m_0+m_1+m_2+m_3\big),%
$$ which implies that
$m_0+m_1+2m_2>\frac{5}{\lambda}$. This shows that $P_4\ne
E_3^4\cap E_4$.
\end{proof}

Thus, the log pair $(S_4,\lambda C^4_d+(\lambda(m_0+m_1+m_2+m_3)-4)E_4)$ is not Kawamata log terminal at $P_4$ and is Kawamata log terminal outside of the point $P_4$.

Let $Z^4$ be the curve in $|2L^4+E_1+2E_2+E_3|$ that passes through the point $P_4$.
Then $Z^4$ is a smooth irreducible curve by Lemma~\ref{lemma:plane-curve-P2-L-E2}.
Denote by $Z$ the proper transform of this curve on $\mathbb{P}^2$. Then $Z$ is a smooth conic in the pencil $\mathcal{L}$ by Lemma~\ref{lemma:plane-curve-pencil}.
If $Z$ is not an irreducible component of the curve $C_d$, then
$$
2d-\big(m_0+m_1+m_2+m_3\big)=Z^4\cdot C_d^4\geqslant \mathrm{mult}_{P_4}(C_d^4).
$$
On the other hand, it follows from Lemma~\ref{lemma:Skoda-2} that
$$
\mathrm{mult}_{P_4}(C_d^4)+m_0+m_1+m_2+m_3>\frac{5}{\lambda}.
$$
This shows that $Z$ is an irreducible component of the curve $C_d$, since $\lambda\leqslant\lambda_2=\frac{5}{2d}$.

Put $C_{d}=Z+C_{d-2}$, where $C_{d-2}$ is a reduced curve in $\mathbb{P}^2$ of degree $d-2$ such that $Z$ is not its irreducible component.
Denote by $C_{d-2}^1$, $C_{d-2}^2$, $C_{d-2}^3$ and $C_{d-2}^4$ its proper transforms on the surfaces $S_1$, $S_2$, $S_3$ and $S_4$, respectively.
Put
$n_0=\mathrm{mult}_{P}(C_{d-2})$,
$n_1=\mathrm{mult}_{P_1}(C_{d-2}^1)$,
$n_2=\mathrm{mult}_{P_2}(C_{d-2}^2)$,
$n_3=\mathrm{mult}_{P_3}(C_{d-2}^3)$ and
$n_4=\mathrm{mult}_{P_4}(C_{d-2}^4)$.
Then
$$
\Big(S_4, \lambda C_{d-2}^4+\lambda Z^4+(\lambda(n_0+n_1+n_2+n_3+4)-4)E_4\Big)
$$
is not Kawamata log terminal at $P_4$ and is Kawamata log terminal outside of the point $P_4$.
Thus, applying Theorem~\ref{theorem:adjunction-2}, we get
$$
\lambda\Big(2\big(d-2\big)-n_0-n_1-n_2-n_3\Big)=\lambda C_{d-2}^4\cdot Z^4>5-\lambda\big(n_0+n_1+n_2+n_3+4\big),%
$$
which implies that $\lambda>\frac{5}{2d}$. This is impossible,
since $\lambda\leqslant\lambda_2=\frac{5}{2d}$.

The obtained
contradiction completes the proof of
Theorems~\ref{theorem:plane-curve} and
\ref{theorem:plane-curve-stability}.

\section{Smooth surfaces in $\mathbb{P}^3$}
\label{section:proof}

The purpose of this section is to prove Theorem~\ref{theorem:main}. 
Let $S$ be a smooth surface in $\mathbb{P}^3$ of degree $d\geqslant 3$, 
let $H_{S}$ be its hyperplane section, let $P$ be a point in $S$, 
let $T_P$ be the hyperplane section of the surface $S$ that is singular at $P$.
Note that $T_P$ is reduced by Lemma~\ref{lemma:Pukhlikov}. 
Put $\lambda=\frac{2d-3}{d(d-2)}$.
Then Theorem~\ref{theorem:main} follows from Theorem~\ref{theorem:plane-curve}, Remark~\ref{remark:convexity} and 

\begin{proposition}
\label{proposition:technical} Let $D$ be any effective
$\mathbb{Q}$-divisor on $S$ such that $D\sim_{\mathbb{Q}} H_S$.
Suppose that $\mathrm{Supp}(D)$ does not contain at least one
irreducible component of the curve $T_P$. Then $(S,\lambda D)$ is
log canonical at $P$.
\end{proposition}

For $d=3$, this result is just \cite[Corollary~1.13]{ChePark13}.
In the remaining part of the section, we will prove Proposition~\ref{proposition:technical}. 
Note that we will do this \emph{without} using \cite[Corollary~1.13]{ChePark13}. Let us start with

\begin{lemma}
\label{lemma:proof-inequalities} The following assertions hold:
\begin{enumerate}
\item[(i)] $\lambda\leqslant \frac{2}{d-1}$,

\item[(ii)] if $d\geqslant 5$, then $\lambda\leqslant\frac{3}{d+1}$,%

\item[(iii)] if $d\geqslant 5$, then $\lambda\leqslant\frac{4}{d+3}$,%

\item[(iv)] If $d\geqslant 6$, then $\lambda\leqslant\frac{3}{d+2}$,%

\item[(v)] $\lambda\leqslant \frac{4}{d+1}$,

\item[(vi)] $\lambda\leqslant \frac{3}{d}$.
\end{enumerate}
\end{lemma}

\begin{proof}
The equality $\frac{2}{d-1}=\lambda+\frac{d-3}{d(d-1)(d-2)}$
implies (i), $\frac{4}{d+1}=\lambda+\frac{d^2-5d+3}{d(d+1)(d-2)}$
implies (ii), and
$\frac{4}{d+3}=\lambda+\frac{2d^2-11d+9}{d(d+3)(d-2)}$ implies
(iii). Similarly, (iv) follows from
$\frac{3}{d+2}=\lambda+\frac{d^2-7d+6}{d(d^2-4)}$, (v) follows
from $\frac{4}{d+1}=\lambda+\frac{2d^2-7d+3}{d(d+1)(d-2)}$, and
(vi) follows from $\frac{3}{d}=\lambda+\frac{d-3}{d(d-2)}$.
\end{proof}

Let $n$ be the number of irreducible components of the curve $T_P$. 
Write 
$$
T_P=T_1+\cdots+T_n,
$$ 
where each $T_i$ is an irreducible curve on the surface $S$. 
For every curve $T_i$, we denote its degree by $d_i$, and we put $t_i=\mathrm{mult}_{P}(T_i)$.

\begin{lemma}
\label{lemma:proof-T-intersections} Suppose that $n\geqslant 2$.
Then
$$
T_i\cdot T_i=-d_i(d-d_i-1)
$$
for every $T_i$, and $T_i\cdot T_j=d_id_j$ for every $T_i$ and $T_j$ such that $T_i\ne T_j$.
\end{lemma}

\begin{proof}
The curve $T_P$ is cut out on $S$ by a hyperplane
$H\subset\mathbb{P}^3$. Then $H\cong\mathbb{P}^2$. Hence, for
every $T_i$ and $T_j$ such that $T_i\ne T_j$, we have $(T_i\cdot
T_j)_S=(T_i\cdot T_j)_H=d_id_j$. In particular, we have
$$
d_1=T_P\cdot T_1=T_1^2+\sum_{i=2}^nT_i\cdot T_1=T_1^2+\sum_{i=2}^n
d_id_1=T_1^2+(d-d_1)d_1,
$$
which gives $T_1\cdot T_1=-d_1(d-d_1-1)$. Similarly, we see that 
$T_i\cdot T_i=-d_i(d-d_i-1)$ for every curve $T_i$.
\end{proof}

Let $D$ be any effective $\mathbb{Q}$-divisor on $S$ such that $D\sim_{\mathbb{Q}} H_S$.
Write
$$
D=\sum_{i=1}^{n}a_iT_i+\Delta,
$$
where each $a_i$ is a non-negative rational number,
and $\Delta$ is an effective $\mathbb{Q}$-divisor on $S$ whose support does not contain
the curves $T_1,\ldots,T_n$.
To prove Proposition~\ref{proposition:technical}, it is enough to show that the log pair $(S,\lambda D)$ is log canonical at $P$
provided that at least one number among $a_1,\ldots,a_n$ vanishes.

Without loss of generality, we may assume that $a_n=0$.
Suppose that the log pair $(S,\lambda D)$ is not log canonical at $P$.
Let us seek for a contradiction.

\begin{lemma}
\label{lemma:proof-Pukhlikov-refined}
Suppose that $n\geqslant 2$. Then
$$
\sum_{i=1}^{k}a_id_id_n\leqslant d_n-t_n\mathrm{mult}_{P}(\Delta).
$$
In particular, $\sum_{i=1}^{k}a_id_i\leqslant 1$ and each $a_i$ does not exceed $\frac{1}{d_i}$.
\end{lemma}

\begin{proof}
One has
$$
d_n=T_n\cdot
D=T_n\cdot\Bigg(\sum_{i=1}^{n}a_iT_i+\Delta\Bigg)=\sum_{i=1}^{n}a_id_id_n+T_n\cdot\Delta\geqslant\sum_{i=1}^{n}a_id_id_n+t_n\mathrm{mult}_{P}(\Delta),
$$
which implies the required inequality.
\end{proof}

Put $m_0=\mathrm{mult}_{P}(D)$.

\begin{lemma}
\label{lemma:proof-Tn-smooth} Suppose that $P\in T_n$. Then $d_n>\frac{d-1}{2}$. If $n\geqslant 2$, then $T_n$ is smooth at $P$.
\end{lemma}

\begin{proof}
Since $T_n$ is not contained in the support of the divisor $D$, we have
$$
d\geqslant d_n=T_n\cdot D\geqslant t_nm_0,
$$
which implies that
$m_0\leqslant\frac{d_n}{t_n}$. Since $m_0>\frac{1}{\lambda}$ by
Lemma~\ref{lemma:Skoda}, we have $d_n>\frac{d-1}{2}$ by
Lemma~\ref{lemma:proof-inequalities}(i). Moreover, if $n\geqslant
2$ and $t_n\geqslant 2$, then it follows from
Lemma~\ref{lemma:Skoda} that
$$
\frac{1}{\lambda}<m_0\leqslant\frac{d_n}{t_n}\leqslant\frac{d-1}{t_n}\leqslant\frac{d-1}{2},
$$
which is impossible by Lemma~\ref{lemma:proof-inequalities}(i).
\end{proof}

Now we are going to use Theorem~\ref{theorem:Trento} to prove

\begin{lemma}
\label{lemma:proof-two-lines}
Suppose that $n\geqslant 3$ and $P$ is contained in at least two irreducible components of the curve $T_P$ that are different from $T_n$ and that are both smooth at $P$.
Then they are tangent to each other at $P$.
\end{lemma}

\begin{proof}
Without loss of generality, we may assume that $P\in T_1\cap T_2$ and $t_1=t_2=1$.
Suppose that $T_1$ and $T_2$ are not tangent to each other at $P$.
Put $\Omega=\sum_{i=3}^{n}a_iT_i+\Delta$, so that $D=a_1T_1+a_2T_2+\Omega$.
Then $a_1d_1+a_2d_2\leqslant 1$ by Lemma~\ref{lemma:proof-Pukhlikov-refined}.

Put $k_0=\mathrm{mult}(\Omega)$. Then
$$
d_1+a_1d_1\big(d-d_1-1\big)-a_2d_1d_2=\Omega\cdot T_1\geqslant k_0%
$$
by Lemma~\ref{lemma:proof-T-intersections}. Similarly, we have
$$
d_2-a_1d_1d_2+a_2d_2\big(d-d_2-1\big)=\Omega\cdot T_2\geqslant k_0.
$$
Adding these two inequalities together and using $a_1d_1+a_2d_2\leqslant 1$, we get
$$
2k_0\leqslant d_1+d_2+\big(a_1d_1+a_2d_2\big)\big(d-d_1-d_2-1\big)\leqslant d_1+d_2+\big(d-d_1-d_2-1\big)=d-1.
$$
Thus, $k_0\leqslant\frac{1}{\lambda}$ by Lemma~\ref{lemma:proof-inequalities}(i).

Since $\lambda k_0\leqslant 1$, we can apply Theorem~\ref{theorem:Trento} to
the log pair $(S, \lambda a_1T_1+\lambda a_2T_2+\lambda\Omega)$ at the point $P$.
This gives either $\lambda\Omega\cdot T_1>2(1-\lambda a_2)$ or $\lambda\Omega\cdot T_2>2(1-\lambda a_1)$.
Without loss of generality, we may assume that $\lambda\Omega\cdot T_2>2(1-\lambda a_1)$. Then
\begin{equation}
\label{equation:lemma:proof-two-lines-1}
d_2+a_2d_2\big(d-d_2-1\big)-a_1d_1d_2=\Omega\cdot T_2>\frac{2}{\lambda}-2a_1.%
\end{equation}
Applying Theorem~\ref{theorem:adjunction-2} to the log pair $(S, \lambda a_1T_1+\lambda b_1T_2+\lambda\Omega)$
and the curve $T_1$ at the point $P$, we get
$$
d_1+a_1d_1\big(d-d_1-1\big)=\Big(\lambda a_2T_2+\lambda\Omega\Big)\cdot T_1>\frac{1}{\lambda}.%
$$
Adding this inequality to \eqref{equation:lemma:proof-two-lines-1}, we get
$$
d+1\geqslant d-1+2a_1\geqslant d_1+d_2+\big(a_1d_1+a_2d_2\big)\big(d-d_1-d_2-1\big)+2a_1>\frac{3}{\lambda},%
$$
because $a_1d_1+a_2d_2\leqslant 1$. Thus, it follows from Lemma~\ref{lemma:proof-inequalities}(ii) that either $d=3$ or $d=4$.

If $d=3$, then $n=3$ and $d_1=d_2=d_3=\lambda=1$, which implies that $a_1+a_2>1$ by \eqref{equation:lemma:proof-two-lines-1}.
On the other hand, we know that $a_1d_1+a_2d_2\leqslant 1$, so that $a_1+a_2\leqslant 1$.
This shows that $d\ne 3$.

We see that $d=4$. Then $\lambda=\frac{5}{8}$ and $d_1+d_2\leqslant 3$.
If $d_1=d_1=1$, then \eqref{equation:lemma:proof-two-lines-1} gives $2a_2+a_1>\frac{11}{5}$.
If $d_1=1$ and $d_2=2$, then \eqref{equation:lemma:proof-two-lines-1} gives $a_2>\frac{3}{5}$.
If $d_1=2$ and $d_2=1$, then \eqref{equation:lemma:proof-two-lines-1} gives $a_2>\frac{11}{5}$.
All these three inequalities are inconsistent, because $a_1d_1+a_2d_2\leqslant 1$.
The obtained contradiction completes the proof of the lemma.
\end{proof}

Note that every line contained in the surfaces $S$ that passes
through $P$ must be an irreducible component of the curve $T_P$.
Moreover, the curve $T_n$ cannot be a line by Lemma~\ref{lemma:proof-Tn-smooth}.
Thus, Lemma~\ref{lemma:proof-two-lines} implies that there exists at most one line in $S$ that passes through $P$.
In particular, we see that $n<d$.

\begin{lemma}
\label{lemma:proof-two-smooth-curves} Suppose that $n\geqslant 3$
and $P$ is contained in at least two irreducible components of the
curve $T_P$ that are different from $T_n$. Then these curves are smooth at $P$.
\end{lemma}

\begin{proof}
Without loss of generality, we may assume that $P\in T_1\cap T_2$ and $t_1\leqslant t_2$.
We have to show that $t_1=t_2=1$.
We may assume that $d\geqslant 5$, because the required assertion is obvious in the cases $d=3$ and $d=4$.

Put $\Omega=\sum_{i=3}^{n}a_iT_i+\Delta$ and put $k_0=\mathrm{mult}_{P}(\Omega)$. Then $m_0=k_0+a_1t_1+a_2t_2$.
Moreover, we have $a_1d_1+a_2d_2\leqslant 1$ by Lemma~\ref{lemma:proof-Pukhlikov-refined}.
On the other hand, it follows from Lemma~\ref{lemma:proof-T-intersections} that
$$
d-1\geqslant d_1+d_2+\big(a_1d_1+a_2d_2\big)\big(d-d_1-d_2-1\big)=\Omega\cdot\Big(T_1+T_2\Big)\geqslant k_0\big(t_1+t_2\big),%
$$
because $a_1d_1+a_2d_2\leqslant 1$.
Thus, we have $k_0\leqslant\frac{d-1}{t_1+t_2}$.
Hence, if $t_1+t_2\geqslant 4$, then
$$
m_0=k_0+a_1t_1+a_2t_2\leqslant k_0+a_1d_1+a_2d_2\leqslant\frac{d-1}{t_1+t_2}+a_1d_1+a_2d_2\leqslant\frac{d-1}{t_1+t_2}+1\leqslant\frac{d+3}{4}
$$
because $a_1d_1+a_2d_2\leqslant 1$.
Since $m_0>\frac{1}{\lambda}$ by Lemma~\ref{lemma:Skoda}, the inequality $m_0\leqslant\frac{d+3}{4}$ gives $\lambda>\frac{d+3}{4}$,
which is impossible by Lemma~\ref{lemma:proof-inequalities}(iii).
Thus, $t_1+t_2\leqslant 3$. Since $t_1\leqslant t_2$, we have $t_1=1$ and $t_2\leqslant 2$.

To complete the proof of the lemma, we have to prove that $t_2=1$.
Suppose $t_2\ne 1$. Then $t_2=2$, since $t_1+t_2\leqslant 3$.
Since $k_0\leqslant\frac{d-1}{t_1+t_2}=\frac{d-1}{3}$ and $a_1d_1+a_2d_2\leqslant 1$, we have
$$
m_0=k_0+a_1t_1+a_2t_2\leqslant k_0+a_1d_1+a_2d_2\leqslant\frac{d-1}{32}+a_1d_1+a_2d_2\leqslant\frac{d-1}{t_1+t_2}+1=\frac{d+2}{3}.%
$$
On the other hand, $m_0>\frac{1}{\lambda}$ by Lemma~\ref{lemma:Skoda}, so that $\lambda>\frac{3}{d+2}$.
Then $d=5$ by Lemma~\ref{lemma:proof-inequalities}(iv).

Since $d=5$, $t_1=1$ and $t_2=2$, we have $n=3$, $d_1=1$, $d_2=3$ and $d_3=1$.
Applying Theorem~\ref{theorem:adjunction-2} to the log pair $(S, \lambda a_1T_1+\lambda a_2T_2+\lambda\Omega)$, we get
$$
1+3a_1=d_1+a_2d_1\big(d-d_1-1\big)=\Big(\lambda a_2T_2+\lambda\Omega\Big)\cdot T_1>\frac{1}{\lambda}=\frac{15}{7},%
$$
which gives $a_1>\frac{8}{21}$.
On the other hand, $a_1+3a_2\leqslant 1$, because $a_1d_1+a_2d_2\leqslant 1$.
Since $m_0>\frac{1}{\lambda}=\frac{15}{7}$ by Lemma~\ref{lemma:Skoda}, we see that
\begin{multline*}
\frac{15}{7}-\frac{1}{9}=\frac{128}{63}>\frac{8-5a_1}{3}=\frac{3-a_1+\frac{7(1-a_1)}{3}}{2}=\frac{3-a_1+7a_2}{2}=\frac{3-3a_1+3a_2}{2}+a_1+2a_2=\\
=\frac{\Delta\cdot T_2}{2}+a_1+2a_2\geqslant\frac{\mathrm{mult}_{P}\Big(\Delta\cdot T_2\Big)}{2}+a_1+2a_2\geqslant\frac{t_2k_0}{2}+a_1+2a_2=k_0+a_1+2a_2=m_0>\frac{15}{7},%
\end{multline*}
which is absurd.
\end{proof}

Now we are ready to prove

\begin{lemma}
\label{lemma:proof-m0-small} One has $m_0\leqslant\frac{d+1}{2}$.
\end{lemma}

\begin{proof}
Suppose that $m_0>\frac{d+1}{2}$. Let us seek for a contradiction.
If $n=1$, then
$$
d=T_n\cdot D\geqslant 2m_0,
$$
which implies that $m_0\leqslant\frac{d}{2}$.
Thus, have $n\geqslant 2$.
Then $a_1\leqslant\frac{1}{d_1}$ by Lemma~\ref{lemma:proof-Pukhlikov-refined}.
Moreover, either $t_n=0$ or $t_n=1$ by Lemma~\ref{lemma:proof-Tn-smooth}.
Hence, there is an irreducible component of $T_P$ that passes through $P$ and is different from $T_n$, because $T_P$ is singular at $P$.
Without loss of generality, we may assume that $t_1\geqslant 1$.

Put $\Upsilon=\sum_{i=2}^{n}a_iT_i+\Delta$, so that $D=a_1T_1+\Upsilon$.
Put $n_0=\mathrm{mult}_{P}(\Upsilon)$, so that $m_0=n_0+a_1t_1$.
Then $t_nn_0\leqslant d_n-a_1d_1d_n$ by Lemma~\ref{lemma:proof-Pukhlikov-refined},
and
\begin{equation}
\label{equation:lemma:proof-m0-small-1}
d_1+a_1d_1(d-d_1-1)=\Upsilon\cdot T_1\geqslant t_1n_0
\end{equation}
by Lemma~\ref{lemma:proof-T-intersections}.
Adding these two inequalities, we get $(t_1+t_n)n_0\leqslant d_1+d_n+a_1d_1(d-d_1-d_n-1)$.
Hence, if $n\geqslant 3$ and $t_n=1$, then
$$
2n_0\leqslant\big(t_1+t_n\big)n_0\leqslant d_1+d_n+a_1d_1\big(d-d_1-d_n-1\big)\leqslant d-1\leqslant d-a_1d_1,%
$$
because $a_1\leqslant\frac{1}{d_1}$. Similarly, if $n=2$ and $t_n=1$, then
$$
2n_0\leqslant\big(t_1+t_n\big)n_0\leqslant d_1+d_2+a_1d_1\big(d-d_1-d_2-1\big)=d_1+d_2-a_1d_1=d-a_1d_1.%
$$
Thus, if $t_n=1$, then $n_0\leqslant\frac{d-a_1d_1}{2}$, which is impossible.
Indeed, the inequality $n_0\leqslant\frac{d-a_1d_1}{2}$ gives
$$
\frac{d+1}{2}<m_0=n_0+a_1t_1\leqslant n_0+a_1d_1\leqslant\frac{d-a_1d_1}{2}+a_1d_1=\frac{d+a_1d_1}{2}\leqslant\frac{d+1}{2},
$$
because $a_1\leqslant\frac{1}{d_1}$. This shows that $t_n=0$.

If $t_1\geqslant 2$, then it follows from \eqref{equation:lemma:proof-m0-small-1} that
$$
\frac{d+1}{2}<m_0\leqslant n_0+a_1d_1\leqslant\frac{d_1+a_1d_1(d-d_1-1)}{2}+a_1d_1=\frac{d_1+a_1d_1(d-d_1+1)}{2}\leqslant\frac{d+1}{2},
$$
because $a_1\leqslant\frac{1}{d_1}$. This shows that $t_1=1$.

Since $t_1=1$ and $t_n=0$, there exists an irreducible component of the curve $T_P$
that passes through $P$  and is different from $T_1$ and $T_n$.
In particular, we have $n\geqslant 3$.
Without loss of generality, we may assume $P\in T_2$.
Then $T_2$ is smooth at $P$ by Lemma~\ref{lemma:proof-two-smooth-curves}.

Put $\Omega=\sum_{i=3}^{n}a_iT_i+\Delta$ and put $k_0=\mathrm{mult}_{P}(\Omega)$.
Then $a_1d_1+a_2d_2\leqslant 1$ by Lemma~\ref{lemma:proof-Pukhlikov-refined}.
Thus, it follows from Lemma~\ref{lemma:proof-T-intersections} that
$$
2k_0\leqslant\Omega\cdot\Big(T_1+T_2\Big)=d_1+d_2+\big(a_1d_1+a_2d_2\big)\big(d-d_1-d_2-1\big)\leqslant d-1,%
$$
which implies $k_0\leqslant\frac{d-1}{2}$. Then
$$
\frac{d+1}{2}<m_0=k_0+a_1t_1+a_2t_2\leqslant k_0+a_1d_1+a_2d_2\leqslant\frac{d-1}{2}+a_1d_1+a_2d_2\leqslant\frac{d-1}{2}+1=\frac{d+1}{2},
$$
because $a_1d_1+a_2d_2\leqslant 1$. The obtained contradiction completes the proof of the lemma.
\end{proof}

Let $f_1\colon S_1\to S$ be a blow up of the point $P$, and let
$E_1$ be its exceptional curve. Denote by $D^1$ the proper
transform of the $\mathbb{Q}$-divisor $D$ on the surface $S_1$.
Then
$$
K_{S_1}+\lambda D^1+\big(\lambda m_0-1\big)E_1\sim_{\mathbb{Q}}
f_1^{*}\Big(K_{S}+\lambda D\Big),
$$
which implies that $(S_1, \lambda D^1+(\lambda m_0-1)E_1)$ is not log canonical at some point $P_1\in E_1$.

By Lemma~\ref{lemma:proof-m0-small}, we have $m_0\leqslant\frac{d+1}{2}$.
By Lemma~\ref{lemma:proof-inequalities}(v), we have $\lambda\leqslant\frac{4}{d+1}$.
This gives $\lambda m_0\leqslant 2$.
Thus, the log pair $(S_1, \lambda D^1+(\lambda m_0-1)E_1)$ is log canonical at every point of the curve $E_1$
that is different from $P_1$ by Lemma~\ref{lemma:log-pull-back}.

Put $m_1=\mathrm{mult}_{P_1}(D^1)$. Then Lemma~\ref{lemma:Skoda} gives
\begin{equation}
\label{equation:proof-Skoda-S1-tak}
m_0+m_1>\frac{2}{\lambda}.
\end{equation}

For each curve $T_i$, denote by $T_i^1$ its proper transform on $S_1$. Put $T_P^1=\sum_{i=1}^nT_i^1$.

\begin{lemma}
\label{lemma:proof-P1-Tp} One has $P_1\not\in T_P^1$.
\end{lemma}

\begin{proof}
Suppose that $P_1\in T_P^1$. Let us seek for a contradiction.
If $T_P$ is irreducible, then
$$
d-2m_0=T_P^1\cdot D^1\geqslant m_1,
$$
so that $m_1+2m_0\leqslant d$.
This inequality gives
$$
\frac{3}{\lambda}<m_1+2m_0\leqslant d,
$$
because $2m_0\geqslant m_0+m_1>\frac{2}{\lambda}$ by \eqref{equation:proof-Skoda-S1-tak}.
This shows that $T_P$ is reducible, because $\lambda\leqslant\frac{3}{d}$ by Lemma~\ref{lemma:proof-inequalities}(vi).

We see that $n\geqslant 2$. If $P_1\in T_n^1$, then
$$
d-1-m_0\geqslant d_n-m_0=d_n-m_0t_n=T_n^1\cdot D^1\geqslant m_1,%
$$
which is impossible, because $m_0+m_1>\frac{2}{\lambda}$ by \eqref{equation:proof-Skoda-S1-tak}, and $\lambda\leqslant\frac{2}{d-1}$ by Lemma~\ref{lemma:proof-inequalities}(i).
Thus, we see that $P_1\not\in T_n^1$.

Without loss of generality, we may assume that $P_1\in T_1^1$.
Put $\Upsilon=\sum_{i=2}^{n}a_iT_i+\Delta$, and
denote by $\Upsilon^1$ the proper transform of the $\mathbb{Q}$-divisor $\Omega$ on the surface $S_1$.
Put $n_0=\mathrm{mult}_{P}(\Upsilon)$, put $n_1=\mathrm{mult}_{P_1}(\Omega^1)$ and put $t_1^1=\mathrm{mult}_{P_1}(T_1^1)$.
Then
$$
d_1+a_1d_1\big(d-d_1-1\big)-n_0t_1=T_1^1\cdot\Upsilon^1\geqslant t_1^1n_1,
$$
which implies that $n_0t_1+n_1t_1^1\leqslant d_1+a_1d_1(d-d_1-1)$.

Note that $t_1^1\leqslant t_1$.
Moreover, we have $a_1\leqslant\frac{1}{d_1}$  by Lemma~\ref{lemma:proof-Pukhlikov-refined}.
Thus, if $t_1^1\geqslant 2$, then
$$
2\big(n_0+n_1\big)\leqslant t_1^1\big(n_0+n_1\big)\leqslant n_0t_1+n_1t_1^1\leqslant d_1+a_1d_1\big(d-d_1-1\big)\leqslant d_1+\big(d-d_1-1\big)=d-1,
$$
which implies that $n_0+n_1\leqslant\frac{d-1}{2}$.
Moreover, if $n_0+n_1\leqslant\frac{d-1}{2}$, then it follows from \eqref{equation:proof-Skoda-S1-tak} that
$$
\frac{d+3}{2}=2+\frac{d-1}{2}\geqslant 2a_1d_1+\frac{d-1}{2}\geqslant 2a_1t_1+\frac{d-1}{2}\geqslant a_1\big(t_1+t_1^1\big)+n_0+n_1=m_0+m_1>\frac{2}{\lambda}%
$$
which gives $d\leqslant 4$ by Lemma~\ref{lemma:proof-inequalities}(iii).
Thus, if $d\geqslant 5$, then $t_1^1=1$.
Furthermore, if $d\leqslant 4$, then $d_1\leqslant 3$, which implies that $t_{1}^1\leqslant 1$.
This shows that $t_1^1=1$ in all cases. Thus, the curve $T_1^1$ is smooth at $P_1$.

Applying Theorem~\ref{theorem:adjunction} to the log pair $(S_1, \lambda\Upsilon^1+\lambda a_1T_1^1+(\lambda(n_0+a_1t_1)-1)E_1)$,
we see that
$$
\lambda\big(d-1-n_0t_1\big)\geqslant\lambda\Big(d_1+a_1d_1\big(d-d_1-1\big)-n_0t_1\Big)=\lambda\Omega^1\cdot T_1^1>2-\lambda\big(n_0+a_1t_1\big),%
$$
because $a_1\leqslant\frac{1}{d_1}$.
Thus, we have $d-1+a_1t_1-n_0(t_1-1)>\frac{2}{\lambda}$.
But $m_0=a_1t_1+n_0>\frac{1}{\lambda}$ by Lemma~\ref{lemma:Skoda}.
Adding these  inequalities together, we obtain
\begin{equation}
\label{equation:lemma:proof-P1-Tp-1}
d-1+2a_1t_1-n_0(t_1-2)>\frac{3}{\lambda}.
\end{equation}
If $t_1\geqslant 2$, this gives
$$
d+1\geqslant d-1+2a_1d_1\geqslant d-1+2a_1t_1\geqslant d-1+2a_1t_1-n_0(t_1-2)>\frac{3}{\lambda}.
$$
because $a_1\leqslant\frac{1}{d_1}$.
One the other hand, if $d\geqslant 5$, then $\lambda\leqslant\frac{3}{d+1}$ by Lemma~\ref{lemma:proof-inequalities}(ii).
Thus, if $d\geqslant 5$, then $t_1=1$.
Moreover, if $d=3$, then $d_1\leqslant 2$, which implies that $t_1=1$ as well.
Furthermore, if $d=4$ and $t_1\ne 1$, then $d_1=3$, $t_1=2$, $\lambda=\frac{5}{8}$, which implies that
$$
\frac{1}{3}=\frac{1}{d_1}\geqslant a_1>\frac{9}{20}
$$
by \eqref{equation:lemma:proof-P1-Tp-1}. Thus, we see that $t_1=1$ in all cases.
This simply means that the curve $T_1$ is smooth at the point $P$.

Since $a_1\leqslant\frac{1}{d_1}$, we have
$$
d-1-n_0\geqslant d_1+a_1d_1\big(d-d_1-1\big)-n_0=\Omega^1\cdot T_1^1\geqslant n_1,%
$$
which implies that $n_1\leqslant\frac{n_0+n_1}{2}\leqslant\frac{d-1}{2}$.
Then $\lambda n_1\leqslant 1$ by Lemma~\ref{lemma:proof-inequalities}(i).
Hence, we can apply Theorem~\ref{theorem:Trento} to the log pair
$(S_1, \lambda\Upsilon^1+\lambda a_1T_1^1+(\lambda(n_0+a_1t_1)-1)E_1)$ at the point $P_1$.
This gives either
$$
\Upsilon^1\cdot T_1^1>\frac{4}{\lambda}-2(n_0+a_1)
$$
or $\Upsilon^1\cdot E_1>\frac{2}{\lambda}-2a_1$ (or both).
Since $a_1\leqslant\frac{1}{d_1}$, the former inequality gives
$$
d-1-n_0\geqslant d_1+a_1d_1\big(d-d_1-1\big)-n_0=\Upsilon^1\cdot T_1^1>\frac{4}{\lambda}-2(n_0+a_1).%
$$
Similarly, the latter inequality gives
$$
n_0=\lambda\Upsilon^1\cdot E_1>\frac{2}{\lambda}-2a_1.
$$
Thus, either $d-1+2a_1+n_0>\frac{4}{\lambda}$ or $2a_1+n_0>\frac{2}{\lambda}$ (or both).

If $t_n\geqslant 1$, then $d_n\ne 1$ by Lemma~\ref{lemma:proof-Tn-smooth}.
Thus, if $t_n\geqslant 1$, then
$$
d-1\geqslant d_n\geqslant a_1d_1d_n+n_0\geqslant 2a_1+n_0
$$
by Lemma~\ref{lemma:proof-Pukhlikov-refined}.
Therefore, if $t_n\geqslant 1$, then
$$
2(d-1)\geqslant d-1+2a+n_0>\frac{4}{\lambda}
$$
or $d-1\geqslant 2a+n_0>\frac{2}{\lambda}$, because $d-1+2a+n_0>\frac{4}{\lambda}$ or $2a+n_0>\frac{2}{\lambda}$.
In both cases, we get $\lambda>\frac{d-1}{2}$, which is impossible by Lemma~\ref{lemma:proof-inequalities}(i).
This shows that $t_n=0$, so that $P\not\in T_n$.

Since $T_1$ is smooth at $P$ and $P\not\in T_n$,
there must be another irreducible component of $T_P$ passing through $P$ that is different from $T_1$ and $T_n$.
In particular, we see that $n\geqslant 3$.
Without loss of generality, we may assume that $P\in T_2$.
Then $T_2$ is smooth at $P$ by Lemma~\ref{lemma:proof-two-smooth-curves}, so that $t_2=1$.
Moreover, the curves $T_1$ and $T_2$ are tangent at $P$ by Lemma~\ref{lemma:proof-two-lines}, which implies that $d\geqslant 4$.
Since $P_1\in T_1^1$, we see that $P_1\in T_2^1$ as well.

Put $\Omega=\sum_{i=3}^{n}a_iT_i+\Delta$ and $k_0=\mathrm{mult}_{P}(\Omega)$, so that $m_0=k_0+a_1+a_2$.
Then $a_1d_1+a_2d_2\leqslant 1$ by Lemma~\ref{lemma:proof-Pukhlikov-refined}.

Denote by $\Omega^1$ the proper transform of the $\mathbb{Q}$-divisor $\Omega$ on the surface $S_1$.
Put $k_1=\mathrm{mult}_{P_1}(\Omega^1)$. Then
$$
d-1-2k_0\geqslant d_1+d_2+\big(a_1d_1+a_2d_2\big)\big(d-d_1-d_2-1\big)-2k_0=\Omega^1\cdot\Big(T_1^1+T_2^1\Big)\geqslant 2k_1%
$$
because $a_1d_1+a_2d_2\leqslant 1$ and $d\geqslant d_1+d_2+d_n\geqslant d_1+d_2+1$.
This gives $k_0+k_1\leqslant\frac{d-1}{2}$.
On the other hand, we have
$$
2a_1+2a_2+k_0+k_1=m_0+m_1>\frac{2}{\lambda}
$$
by \eqref{equation:proof-Skoda-S1-tak}. Thus, we have
$$
\frac{d+3}{2}=2+\frac{d-1}{2}\geqslant
2\big(a_1d_1+a_2d_2\big)+\frac{d-1}{2}\geqslant
2a_1+2a_2+\frac{d-1}{2}\geqslant 2a_1+2a_2+k_0+k_1>\frac{2}{\lambda}
$$
because $a_1d_1+a_2d_2\leqslant 1$.
By Lemma~\ref{lemma:proof-inequalities}(iii) this gives $d=4$. Thus, we have $\lambda=\frac{5}{8}$.

Since $d=4>n\geqslant 3$, we have $n=3$.
Without loss of generality, we may assume that $d_1\leqslant d_2$.
By Lemma~\ref{lemma:proof-two-lines}, there exists at most one line in $S$ that passes through $P$.
This shows that $d_1=1$, $d_2=2$ and $d_3=1$. Thus, $T_1$ and $T_3$ are lines, $T_2$ is a conic,
$T_1$ is tangent to $T_2$ at $P$, and $T_3$ does not pass through $P$.
In particular, the curves $T_1^1$ and $T_1^2$ intersect each other transversally at $P_1$.

By Lemma~\ref{lemma:proof-T-intersections}, we have $T_1\cdot T_1=T_2\cdot T_2=-2$ and $T_1\cdot T_2=2$.
On the other hand, the log pair $(S_1, \lambda a_1 T_1^1+\lambda a_2 T_2^1+\lambda\Omega^1+(\lambda(a_1+a_2+k_0)-1)E_1)$
is not log canonical at the point $P_1$.
Thus, applying Theorem~\ref{theorem:adjunction} to this log pair and the curve $T_1^1$, we get
$$
\lambda\big(1+2a_1-2a_2-k_0\big)=\lambda\Omega^1\cdot T_1^1>2-\lambda(a_1+a_2+k_0)-\lambda a_2,%
$$
which implies that $3a_1>\frac{2}{\lambda}-1=\frac{11}{5}$, because $\lambda=\frac{5}{8}$.
Similarly, applying Theorem~\ref{theorem:adjunction} to this log pair and the curve $T_2^1$, we get
$$
\lambda\big(2-2a_1+2a_2-k_0\big)=\lambda\Omega^1\cdot T_2^1>2-\lambda(a_1+a_2+k_0)-\lambda a_1,%
$$
which implies that $3a_2>\frac{2}{\lambda}-2=\frac{6}{5}$.
Hence, we have $a_1>\frac{11}{15}$ and $a_2>\frac{2}{5}$, which is impossible, since $a_1+2a_2=a_1d_1+a_2d_2\leqslant 1$.
The obtained contradiction completes the proof of the lemma.
\end{proof}

Now we are going to show that the  curve $T_P$ has at most two irreducible components.
This follows from

\begin{lemma}
\label{lemma:proof-Tp-mult-2} One has $n\geqslant 2$ and $\mathrm{mult}_{P}(T_P)=2$.
Moreover, if $n=2$, then $P\in T_1\cap T_2$, both curves $T_1$ and $T_2$ are smooth at $P$, and $d_1\leqslant d_2$.
\end{lemma}

\begin{proof}
If $T_P$ is irreducible and $\mathrm{mult}_{P}(T_P)\geqslant 3$, then Lemma~\ref{lemma:Skoda} gives
$$
d=T_P\cdot D\geqslant 3m_0>\frac{3}{\lambda},
$$
which is impossible by Lemma~\ref{lemma:proof-inequalities}(vi).
Thus, if $n=1$, then $\mathrm{mult}_{P}(T_P)=2$.

To complete the proof, we may assume that $n\geqslant 2$.
Then $t_n=0$ or $t_n=1$ by Lemma~\ref{lemma:proof-Tn-smooth}.
In particular, there exists an irreducible component of the curve $T_P$ different from $T_n$ that passes through $P$.
Without loss of generality, we may assume that $P\in T_1$.

Put $\Upsilon=\sum_{i=2}^{n}a_iT_i+\Delta$,
and denote by $\Upsilon^1$ the proper transform of the $\mathbb{Q}$-divisor $\Omega$ on the surface $S_1$.
Put $n_0=\mathrm{mult}_{P}(\Upsilon)$.
Then the log pair $(S_1, \lambda\Upsilon^1+(\lambda(n_0+a_1t_1)-1)E_1)$ is not log canonical at $P_1$, since $P_1\not\in T_1^1$ by Lemma~\ref{lemma:proof-P1-Tp}.
In particular, it follows from Theorem~\ref{theorem:adjunction-2} that
$$
\lambda n_0=\lambda\Upsilon^1\cdot E_1>1,
$$
which implies that $n_0>\frac{1}{\lambda}$.
Thus, if $t_1\geqslant 2$, then it follows from Lemma~\ref{lemma:proof-T-intersections} that
$$
\frac{1}{\lambda}\geqslant\frac{d-1}{2}\geqslant \frac{d_1+a_1d_1(d-d_1-1)}{2}=\frac{\Upsilon\cdot T_1}{2}\geqslant \frac{t_1n_0}{2}\geqslant n_0>\frac{1}{\lambda},%
$$
because $a_1\leqslant\frac{1}{d_1}$ by Lemma~\ref{lemma:proof-Pukhlikov-refined}, and $\lambda\leqslant\frac{2}{d-1}$ by Lemma~\ref{lemma:proof-inequalities}(i).
This shows that $t_1=1$, so that the curve $T_1$ is smooth at $P$.

If $t_n=1$ and $n\geqslant 3$, then
$$
\frac{2}{\lambda}\geqslant d-1\geqslant d_1+d_n+ad_1(d-d_1-d_n-1)=\Upsilon\cdot\Big(T_1+T_n\Big)\geqslant 2n_0>\frac{2}{\lambda}.
$$
Thus, if $t_n=1$, then $n=2$. Vice versa, if $n=2$, then $t_n=1$, because $T_1$ is smooth at $P$.
Furthermore, if $n=2$, then $d_1\leqslant d_n$, because $d_n>\frac{d-1}{2}$ by Lemma~\ref{lemma:proof-Tn-smooth}.
Therefore, to complete the proof, we must show that $n=2$.

Suppose that $n\geqslant 3$. Let us seek for a contradiction.
We know that $P\not\in T_n$, so that $t_n=0$.
Then every irreducible component of the curve $T_P$ that contain $P$ is smooth at $P$ by Lemma~\ref{lemma:proof-two-smooth-curves}.
Hence, there should be at least one irreducible component of the curve $T_P$ containing $P$ that is different from $T_1$ and $T_n$.
Without loss of generality, we may assume that $P\in T_2$.

Put $\Omega=\sum_{i=3}^{n}a_iT_i+\Delta$ and $k_0=\mathrm{mult}_{P}(\Omega)$.
By Lemma~\ref{lemma:proof-Pukhlikov-refined}, we have $a_1d_1+a_2d_2\leqslant 1$.
Thus, it follows from Lemma~\ref{lemma:proof-T-intersections} that
$$
2k_0\leqslant\Delta\cdot\Big(T_1+T_2\Big)=d_1+d_2+\big(a_1d_1+a_2d_2\big)\big(d-d_1-d_2-1\big)\leqslant d_1+d_2+\big(d-d_1-d_2-1\big)=d-1.
$$
Hence, we have $k_0\leqslant\frac{d-1}{2}$.

Denote by $\Omega^1$ the proper transform of the $\mathbb{Q}$-divisor $\Omega$ on the surface $S_1$.
Then the log pair $(S_1, \lambda\Omega^1+(\lambda(k_0+a_1+a_2)-1)E_1)$ is not log canonical at $P_1$,
because $P_1\not\in T_1^1$ and $P_1\not\in T_2^1$ by Lemma~\ref{lemma:proof-P1-Tp}.
In particular, it follows from Theorem~\ref{theorem:adjunction} that
$$
\lambda k_0=\lambda\Omega^1\cdot E_1>1,
$$
which implies that $k_0>\frac{1}{\lambda}$.
This contradicts Lemma~\ref{lemma:proof-inequalities}(i), because $k_0\leqslant\frac{d-1}{2}$.
\end{proof}

Later, we will need the following simple

\begin{lemma}
\label{lemma:proof-mult-quartic-11-5} Suppose that $d=4$. Then
$m_0\leqslant\frac{11}{5}$.
\end{lemma}

\begin{proof}
If $n=1$, then
$$
2t_n\geqslant d_n=T_n\cdot D\geqslant t_nm_0,
$$
so that $m_0\leqslant 2<\frac{11}{5}$.
Thus, we may assume that $n\ne 1$.
Then it follows from Lemma~\ref{lemma:proof-Tp-mult-2} that $n=2$,
$P\in T_1\cap T_2$, both curves $T_1$ and $T_2$ are smooth at $P$, and $d_1\leqslant d_2$.

If $d_2=2$, then $m_0\leqslant 2<\frac{11}{5}$, because
$$
2=T_2\cdot D\geqslant m_0.
$$
Thus, we may assume that $d_2\ne 2$. Then $d_1=1$ and $d_2=3$.
Then $\mathrm{mult}_{P}(\Delta)+3a_1\leqslant 3$ by Lemma~\ref{lemma:proof-Pukhlikov-refined}.
Moreover, we have
$$
1+2a_1=T_1\cdot\Delta\geqslant\mathrm{mult}_{P}(\Delta).
$$
The obtained inequalities give $m_0=\mathrm{mult}_{P}(\Delta)+a_1\leqslant\frac{11}{5}$.
\end{proof}

Let $f_2\colon S_2\to S_1$ be a blow up of the point $P_1$. Denote
by $E_2$ the $f_2$-exceptional curve, denote by $E_1^2$ the proper
transform of the curve $E_1$ on the surface $S_2$, and denote by
$D^2$ the proper transform of the $\mathbb{Q}$-divisor $D$ on the
surface $S_2$. Then
$$
K_{S_2}+\lambda D^2+\big(\lambda
m_0-1\big)E_1^2+\Big(\lambda\big(m_0+m_1\big)-2\Big)E_2\sim_{\mathbb{Q}}
f_2^*\Big(K_{S_1}+\lambda D^1+\big(\lambda m_0-1\big)E_1\Big).
$$
By Remark~\ref{remark:log-pull-back}, the log pair $(S_2, \lambda
D^2+(\lambda m_0-1)E_1^2+(\lambda(m_0+m_1)-2)E_2)$ is not log
canonical at some point $P_2\in E_1$.

\begin{lemma}
\label{lemma:proof-mult-S-2} One has
$m_0+m_1\leqslant\frac{3}{\lambda}$.
\end{lemma}

\begin{proof}
Suppose that $m_0+m_1>\frac{3}{\lambda}$.
Then $2m_0\geqslant m_0+m_1>\frac{3}{\lambda}$.
But $m_0\leqslant\frac{d+1}{2}$ by Lemma~\ref{lemma:proof-m0-small}.
Then $\lambda>\frac{3}{d+1}$.
Thus, we have $d\leqslant 4$ by Lemma~\ref{lemma:proof-inequalities}(ii).
Moreover, if $d=4$, then
$$
\frac{22}{5}\geqslant 2m_0\geqslant m_0+m_1>\frac{3}{\lambda}=\frac{24}{5}%
$$
by Lemma~\ref{lemma:proof-mult-quartic-11-5}. This shows that $d=3$.

We have $\lambda=1$. If $n=1$, then
$$
3=T_P\cdot D\geqslant 2m_0\geqslant m_1+m_0>\frac{3}{\lambda}=3,
$$
which is absurd. Hence, it follows from Lemma~\ref{lemma:proof-Tp-mult-2} that $n=2$,
$d_1=1$, $d_2=2$ and $P\in T_1\cap T_2$.

We have $m_0=\mathrm{mult}_{P}(\Delta)+a_1$.
On the other hand, we have $\mathrm{mult}_{P}(\Delta)+2a_1\leqslant 2$ by Lemma~\ref{lemma:proof-Pukhlikov-refined}.
Moreover, we have
$$
1+a_1=T_1\cdot\Omega\geqslant \mathrm{mult}_{P}(\Delta),
$$
which implies that $\mathrm{mult}_{P}(\Delta)-a_1\leqslant 1$.
Adding these inequalities, we get
$$
3\geqslant 2\mathrm{mult}_{P}(\Delta)+a=\mathrm{mult}_{P}(\Delta)+m_0\geqslant m_1+m_0>\frac{3}{\lambda}=3,
$$
because $\mathrm{mult}_{P}(\Delta)\geqslant m_1$, since $P_1\not\in T_1^1$ by Lemma~\ref{lemma:proof-P1-Tp}.
\end{proof}

Thus, the log pair  $(S_2, \lambda D^2+(\lambda m_0-1)E_1^2+(\lambda(m_0+m_1)-2)E_2)$ is log canonical at every point of the curve $E_2$
that is different from the point $P$ by Lemma~\ref{lemma:log-pull-back}.

\begin{lemma}
\label{lemma:proof-E1-E2} One has $P_2\ne E_1^2\cap E_2$.
\end{lemma}

\begin{proof}
Suppose that $P_2=E_1^2\cap E_2$. Then Theorem~\ref{theorem:adjunction} gives
$$
\lambda\big(m_0-m_1\big)=\lambda D^2\cdot E_1^2>3-\lambda\big(m_0+m_1\big),%
$$
which implies that $m_0>\frac{3}{2\lambda}$.
But $m_0\leqslant\frac{d+1}{2}$ by Lemma~\ref{lemma:proof-m0-small}.
Therefore, we have $\lambda>\frac{3}{d+1}$, which implies that $d\leqslant 4$ by Lemma~\ref{lemma:proof-inequalities}(ii).
If $d=4$, then
$$
\frac{12}{5}=\frac{3}{2\lambda}<m_0\leqslant\frac{11}{5}
$$
by Lemma~\ref{lemma:proof-mult-quartic-11-5}. Thus, we have $d=3$.

One has $\lambda=1$. If $n=1$, then
$$
3=T_P\cdot D\geqslant 2m_0>\frac{3}{\lambda}=3,
$$
which is absurd.
Hence, it follows from Lemma~\ref{lemma:proof-Tp-mult-2} that $n=2$, $d_1=1$, $d_2=2$ and $P\in T_1\cap T_2$.

We have $m_0=\mathrm{mult}_{P}(\Delta)+a_1$.
Moreover, we have $\mathrm{mult}_{P}(\Delta)+2a_1\leqslant 2$ by Lemma~\ref{lemma:proof-Pukhlikov-refined},
Then $2\mathrm{mult}_{P}(\Delta)+a_1\leqslant 3$, because
$$
1+a_1=T_1\cdot\Delta\geqslant \mathrm{mult}_{P}(\Delta).
$$

Denote by $\Delta^1$ the proper transform of the divisor $\Delta$ on the surface $S_1$,
and denote by $\Delta^2$ the proper transform of the divisor $\Delta$ on the surface $S_2$.
Then $m_1=\mathrm{mult}_{P_1}(\Delta^1)$, because $P_1\not\in T_1^1$ by Lemma~\ref{lemma:proof-P1-Tp}.
Thus, the log pair $(S_2, \lambda\Delta^2+(m_0-1)E_1^2+(m_0+m_1-2)E_2)$ is not log canonical at $P_2$.
Applying Theorem~\ref{theorem:adjunction} to this pair and the curve $E_1^2$, we get
$$
\mathrm{mult}_{P}(\Delta)-m_1=\Delta^2\cdot E_1^2>3-m_0-m_1,%
$$
which implies that $2\mathrm{mult}_{P}(\Delta)+a_1>3$.
The latter is impossible, because we already proved that $2\mathrm{mult}_{P}(\Delta)+a_1\leqslant 3$.
\end{proof}

Thus, the log pair $(S_2, \lambda D^2+(\lambda(m_0+m_1)-2)E_2)$ is not log canonical at $P_2$. Then Lemma~\ref{lemma:Skoda} gives

\begin{equation}
\label{equation:proof-Skoda-S-2}
m_0+m_1+m_2>\frac{3}{\lambda}.
\end{equation}

Denote by $T_P^2$ the proper transform of the curve $T_P$ on the
surface $S^2$. Then
$$
T_P^2+E_1^2\sim (f_{1}\circ f_2)^*(\mathcal{O}_{S}(1))-f_2^*(E_1)-E_2,%
$$
because $T_P^1\sim f_{1}^*(\mathcal{O}_{S}(1))-2E_1$ by
Lemma~\ref{lemma:proof-Tp-mult-2}, and $P_1\not\in T_P^1$ by
Lemma~\ref{lemma:proof-P1-Tp}.

\begin{lemma}
\label{lemma:proof-pencil} The linear system $|T_P^2+E_1^2|$ is a
pencil that does not have base points in $E_2$.
\end{lemma}

\begin{proof}
Since $|T_P^1+E_1|$ is a two-dimensional linear system that does
not have base points, $|T_P^2+E_1^2|$ is a pencil. Let $C$ be a
curve in $|T_P^1+E_1|$ that passes through $P_1$ and is different
from $T_P^1+E_1$. Then $C$ is smooth at $P$, since $P\in f_1(C)$
and $f_1(C)$ is a hyperplane section of the surface $S$ that is
different from $T_P$. Since $C\cdot E_1=1$, we see that $T_P^1+E_1$
and $C$ intersect transversally at $P_1$. Thus, the proper
transform  of the curve $C$ on the surface $S_2$ is contained in
$|T_P^1+E_1|$ and have no common points with $T_P^2+E_1^2$ in
$E_2$. This shows that the pencil $|T_P^1+E_1|$ does not have base
points in $E_2$.
\end{proof}

Let $Z^2$ be the curve in $|T_P^2+E_2|$ that passes through the point $P_2$.
Then
$$
Z^2\ne T_P^2+E_1^2,
$$
because $P_2\ne E_1^2\cap
E_2$ by Lemma~\ref{lemma:proof-E1-E2}. Then $Z_2$ is smooth at
$P_2$. Put $Z=f_1\circ f_2(Z^2)$ and $Z^1=f_2(Z^2)$. Then $P\in Z$
and $P_1\in Z^1$. Moreover, the curve $Z$ is smooth at $P$, and
the curve $Z_1$ is smooth at $P_1$. Furthermore, the curve $Z$ is
reduced by Lemma~\ref{lemma:Pukhlikov}.

The log pair $(S,\lambda Z)$ is log canonical at $P$, because $Z$ is smooth at $P$. 
Note that 
$$
Z\sim_{\mathbb{Q}} D.
$$ 
Thus, we may assume that $\mathrm{Supp}(D)$ does not contain at least one irreducible component of the curve $Z$ by Remark~\ref{remark:convexity}. 
Denote this irreducible component by $\overline{Z}$, and denote its degree in $\mathbb{P}^3$ by $\bar{d}$.
Then $\bar{d}\leqslant d$.

\begin{lemma}
\label{lemma:proof-bar-Z-does-not-contain-P} One has
$P\not\in\overline{Z}$.
\end{lemma}

\begin{proof}
Suppose that $P\in\overline{Z}$. Let us seek for a contradiction.
Denote by $\overline{Z}^2$ the proper transform of the curve $\overline{Z}$
on the surface $S_2$. Then
$$
d-m_0-m_1\geqslant\bar{d}-m_0-m_1=\overline{Z}^2\cdot D^2\geqslant m_2,
$$
which implies that $m_0+m_1+m_2\leqslant d$. One the other hand,
$m_0+m_1+m_2>\frac{3}{\lambda}$ by
\eqref{equation:proof-Skoda-S-2}.  This gives
$\lambda>\frac{3}{d}$, which is impossible by
Lemma~\ref{lemma:proof-inequalities}(vi).
\end{proof}

In particular, the curve $Z$ is reducible.
Denote by $\widehat{Z}$ its irreducible component that passes through $P$,
denote its proper transform on the surface $S_1$ by $\widehat{Z}^1$,
and denote its proper transform on the surface $S_2$ by $\widehat{Z}^2$.
Then $\overline{Z}\ne\widehat{Z}$,
$P_1\in\widehat{Z}^1$ and $P_2\in\widehat{Z}^2$. Denote by $\hat{d}$ the
degree of the curve $\widehat{Z}$ in $\mathbb{P}^3$. Then
$\hat{d}+\bar{d}\leqslant d$. Moreover, the intersection form of
the curves $\widehat{Z}$ and $\overline{Z}$ on the surface $S$ is given by

\begin{lemma}
\label{lemma:proof-Z-intersections} One has
$\overline{Z}\cdot\overline{Z}=-\bar{d}(d-\bar{d}-1)$,
$\widehat{Z}\cdot\widehat{Z}=-\hat{d}(d-\hat{d}-1)$ and
$\overline{Z}\cdot\widehat{Z}=\bar{d}\hat{d}$.
\end{lemma}

\begin{proof}
See the proof of Lemma~\ref{lemma:proof-T-intersections}.
\end{proof}

Put $D=a\widehat{Z}+\Omega$, where $a$ is a positive rational number,
and $\Omega$ is an effective $\mathbb{Q}$-divisor on the surface
$S$ whose support does not contain the curve $\widehat{Z}$. Denote by
$\Omega^1$ the proper transform of the divisor $\Omega$ on the
surface $S_1$, and denote by $\Omega^2$ the proper transform of
the divisor $\Omega$ on the surface $S_2$. Put
$n_0=\mathrm{mult}_{P}(\Omega)$,
$n_1=\mathrm{mult}_{P_1}(\Omega^1)$ and
$n_2=\mathrm{mult}_{P_2}(\Omega^2)$. Then $m_0=n_0+a$, $m_1=n_1+a$
and $m_2=n_2+a$. Then the log pair $(S_2, \lambda
a\widehat{Z}^2+\lambda\Omega^2+(\lambda(n_0+n_1+2a)-2)E_2)$ is not log
canonical at $P_2$, because  $(S_2, \lambda
D^2+(\lambda(m_0+m_1)-2)E_2)$ is not log canonical at $P_2$. Thus,
applying Theorem~\ref{theorem:adjunction}, we see that
$$
\lambda\Big(\Omega\cdot\widehat{Z}-n_0-n_1\Big)=\lambda\Omega^2\cdot
Z^2>1-\Big(\lambda\big(n_0+n_1+2a\big)-2\Big)=3-\lambda\big(n_0+n_1+2a\big),
$$
which implies that
\begin{equation}
\label{equation:proof-end-1}
\Omega\cdot\widehat{Z}>\frac{3}{\lambda}-2a.
\end{equation}
On the other hand, we have
$$
\bar{d}=D\cdot\overline{Z}=\Big(a\widehat{Z}+\Omega\Big)\cdot\overline{Z}\geqslant a\widehat{Z}\cdot\overline{Z}=a\hat{d}\bar{d}%
$$
by
Lemma~\ref{lemma:proof-Z-intersections}. This gives
\begin{equation}
\label{equation:proof-end-2} a\leqslant\frac{1}{\hat{d}}.
\end{equation}
Thus, it follows from \eqref{equation:proof-end-1},
\eqref{equation:proof-end-2} and
Lemma~\ref{lemma:proof-Z-intersections} that
$$
\frac{3}{\lambda}-2\leqslant\frac{3}{\lambda}-2a<\Omega\cdot\widehat{Z}=\hat{d}+a\hat{d}\Big(d-\hat{d}-1\Big)\leqslant d-1,%
$$
which implies that $\lambda>\frac{3}{d+1}$. Then $d\leqslant 4$ by
Lemma~\ref{lemma:proof-inequalities}(ii).

\begin{lemma}
\label{lemma:proof-end-quartic} One has $d\ne 4$.
\end{lemma}

\begin{proof}
Suppose that $d=4$. Then $\lambda=\frac{5}{8}$ and $\hat{d}\leqslant 3$. By
Lemma~\ref{lemma:proof-P1-Tp}, $\widehat{Z}$ is not a line, since
every line passing through $P$ must be an irreducible component of
the curve $T_P$. Thus, either $\widehat{Z}$ is a conic or $\widehat{Z}$ is
a plane cubic curve. If $\widehat{Z}$ is a conic, then $\widehat{Z}^2=-2$
and $a\leqslant\frac{1}{2}$ by \eqref{equation:proof-end-2}. Thus,
if $\widehat{Z}$ is a conic, then
$$
2+2a=\Omega\cdot\widehat{Z}>\frac{3}{\lambda}-2a=\frac{24}{5}-2a,
$$
which implies that $\frac{1}{2}\geqslant a>\frac{7}{10}$. This
shows that $\widehat{Z}$ is a plane cubic curve. Then $\widehat{Z}^2=0$.
Since $a\leqslant\frac{1}{3}$ by \eqref{equation:proof-end-2}, we have
$$
3=\Omega\cdot\widehat{Z}>\frac{3}{\lambda}-2a=\frac{24}{5}-2a\geqslant \frac{24}{5}-\frac{2}{3}=\frac{62}{15},%
$$ which is absurd.
\end{proof}

Thus, we see that $d=3$. Then $\widehat{Z}$ us either a line or a conic.
But every line passing through $P$ must be an irreducible component of $T_P$. Since
$\widehat{Z}$ is not an irreducible component of $T_P$ by
Lemma~\ref{lemma:proof-P1-Tp}, the curve $\widehat{Z}$ must be a
conic. Then $\widehat{Z}^2=0$.
Therefore, it follows from \eqref{equation:proof-end-1} that
$$
3-2a=\frac{3}{\lambda}-2a<\Omega\cdot\widehat{Z}=\hat{d}+a\hat{d}\Big(d-\hat{d}-1\Big)=\hat{d}=2,%
$$
which implies that $a>\frac{1}{2}$. But
$a\leqslant\frac{1}{\hat{d}}=\frac{1}{2}$ by
\eqref{equation:proof-end-2}. The obtained contradiction completes
the proof of Theorem~\ref{theorem:main}.

\end{document}